\documentclass[a4paper,reqno,11pt]{amsart} 
\usepackage[T1]{fontenc}
\usepackage[utf8]{inputenc}

\usepackage{amsmath, amssymb, graphicx, color}
\usepackage[dvipsnames,svgnames,table]{xcolor}
\usepackage[foot]{amsaddr}
\usepackage[inline]{enumitem}
\setenumerate{label=\textup{(\roman*)}, noitemsep, topsep=3pt-\parskip, 
labelindent=.2em, leftmargin=*, widest=iii,}

\usepackage[x11names]{xcolor}
\usepackage{tikz}
\usetikzlibrary{arrows.meta}

\usepackage{textpos} 

\makeatletter
\def\thm@space@setup{
  \thm@preskip=4mm
  \thm@postskip=0mm
}
\makeatother

\usepackage{hyperref}
\hypersetup{
  pdfauthor={Maël Dumas, Jędrzej Hodor, Tomasz Mazur, Piotr Micek, Clément Rambaud, Giannos Stamoulis},
  pdftitle={Invitation to torso parameters of annotated graphs},
    colorlinks,
    linkcolor={RoyalBlue},
    citecolor={RubineRed},
    urlcolor={blue!80!black}
}
\usepackage[capitalise, noabbrev]{cleveref}
\usepackage{mathtools}

\theoremstyle{plain} 

\newtheorem{theorem}{Theorem}
\newtheorem{corollary}[theorem]{Corollary} 
\newtheorem{problem}[theorem]{Problem}
\newtheorem{conjecture}[theorem]{Conjecture}
\newtheorem{lemma}[theorem]{Lemma}
\newtheorem{obs}[theorem]{Observation}

\newtheorem*{theorem*}{Theorem}
\newtheorem*{corollary*}{Corollary} 
\newtheorem*{problem*}{Problem}
\newtheorem*{conjecture*}{Conjecture}
\newtheorem*{lemma*}{Lemma}
\newtheorem*{obs*}{Observation}
\newtheorem*{remark*}{Remark}
\newtheorem*{proposition*}{Proposition}
\newtheorem*{question*}{Question} 
\newtheorem*{example*}{Example}

\theoremstyle{remark} 
\newtheorem{claim}[theorem]{Claim} 
\newtheorem*{claim*}{Claim} 
 
\newtheorem*{definition*}{Definition} 
\bibstyle{plain}

\usepackage{thmtools, thm-restate}

\newenvironment{proofclaim}[1][]
    {\let\oldqed\qedsymbol\renewcommand{\qedsymbol}{\ensuremath{\lozenge}}\begin{proof}[Proof of the claim] }{\end{proof}\renewcommand{\qedsymbol}{\oldqed}}

\newcommand{\q}[1]{``#1''}

\colorlet{defcolor}{ForestGreen}
\newcommand{\defin}[1]{\relax\ifmmode{\color{defcolor}{#1}}\else{\emph{\textcolor{defcolor}{#1}}}\fi}

\newcommand{\gm}{f_{\boxplus}}

\newcommand{\N}{\mathbb{N}}
\newcommand{\posint}{\N}

\newcommand{\Oh}{\mathcal{O}}

\newcommand{\calB}{\mathcal{B}} 
\newcommand{\calC}{\mathcal{C}}
\newcommand{\calD}{\mathcal{D}}
 
\newcommand{\calF}{\mathcal{F}}
 
\newcommand{\calH}{\mathcal{H}}

\newcommand{\calM}{\mathcal{M}}

\newcommand{\calP}{\mathcal{P}} 
\newcommand{\calQ}{\mathcal{Q}}

\newcommand{\calT}{\mathcal{T}}

\newcommand{\calW}{\mathcal{W}}

\DeclareMathOperator\ep{ep}
\DeclareMathOperator\tw{tw}

\DeclareMathOperator\pw{pw}
\DeclareMathOperator\td{td}
\DeclareMathOperator\vc{vc}
\DeclareMathOperator\fv{fv}
\DeclareMathOperator\had{had}
\DeclareMathOperator\tdtwo{td_2}

\newcommand{\param}{\mathrm{p}}
\newcommand{\torso}{\mathrm{torso}}

\newcommand{\torsobullet}{\bullet}
\newcommand{\simono}{\mathrm{s}}

\newcommand{\bigO}{\mathcal{O}}

\newcommand{\stpath}[2]{$#1$--$#2$ path}
\newcommand{\stpaths}[2]{$#1$--$#2$ paths}

\let\le\leqslant
\let\ge\geqslant
\let\leq\leqslant
\let\geq\geqslant

\let\subset\subseteq

\let\epsilon\varepsilon

\let\setminus\backslash

\graphicspath{{figs/}}

\makeatletter 
\let\old@setaddresses\@setaddresses 
\def\@setaddresses{\bigskip\bgroup\parindent 0pt\let\scshape\relax\old@setaddresses\egroup}
\makeatother

\crefformat{equation}{#2(#1)#3}
\let\eqref\cref
\crefformat{subsection}{Subsection #2#1#3}

\begin{document} 
\title[Invitation to torso parameters of annotated graphs]{Invitation to torso parameters of annotated graphs}

\author[Dumas]{Ma\"{e}l Dumas}
\address[Dumas, Stamoulis]{Université Paris Cité, CNRS, IRIF, F-75013, Paris, France}
\email{mael.dumas@irif.fr}

\author[Hodor]{J\k{e}drzej Hodor}
\address[Hodor]{Theoretical Computer Science Department, 
Faculty of Mathematics and Computer Science and  Doctoral School of Exact and Natural Sciences, Jagiellonian University, Krak\'ow, Poland}
\email{jedrzej.hodor@gmail.com}

\author[Mazur]{Tomasz Mazur}
\address[Mazur, Micek]{Theoretical Computer Science Department, 
Faculty of Mathematics and Computer Science, Jagiellonian University, Krak\'ow, Poland}
\email{tom.mazur34@gmail.com}

\author[Micek]{Piotr Micek}
\email{piotr.micek@uj.edu.pl}

\author[Rambaud]{Clément Rambaud}
\address[Rambaud]{Mathematics department, TU Darmstadt, Germany}
\email{clement.rambaud@normalesup.org}

\author[Stamoulis]{Giannos Stamoulis}
\email{stamoulis@irif.fr}

\thanks{M.\ Dumas received funding from ERC grant BUKA (No.~101126229). J.\ Hodor is supported by a Polish Ministry of Education and Science grant (Perły Nauki; PN/01/0265/2022). 
T.\ Mazur and P.\ Micek were partially supported by the National Science Centre of Poland under grant UMO-2023/05/Y/ST6/00079 within the Weave-UNISONO program.
C.\ Rambaud has been funded by the European Union through European Research Council project WYDRA, grant agreement no. 101230549. Views and opinions expressed are however those of the authors only and do not necessarily reflect those of the European Union or the European Research Council Executive Agency. Neither the European Union nor the granting authority can be held responsible for them.
G.\ Stamoulis was supported by the project BOBR that is funded from the European Research Council (ERC) under the European Union’s Horizon 2020 research and innovation programme with grant agreement No. 948057. In particular, part of work on this manuscript was done while G.\ Stamoulis was affiliated with University of Warsaw.}

\begin{textblock}{20}(12.3, 11.2)
\includegraphics[width=35px]{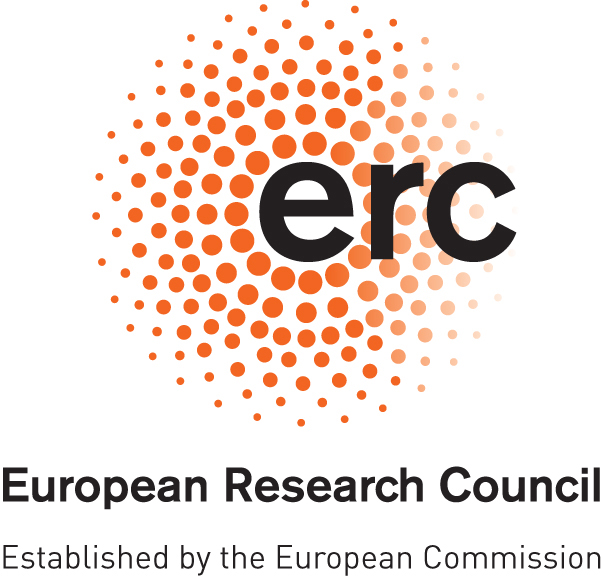}%
\end{textblock}
\begin{textblock}{20}(12.3, 11.9)
\includegraphics[width=35px]{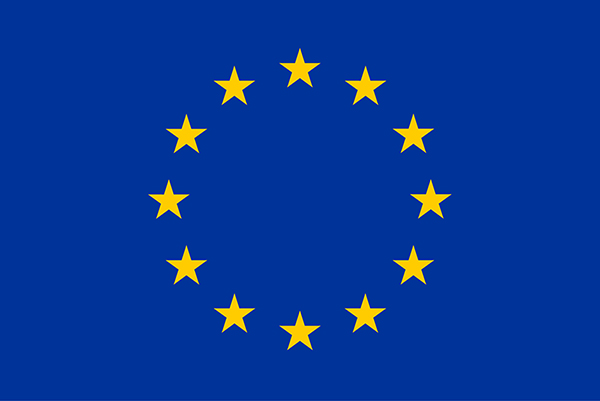}%
\end{textblock}

\begin{abstract}
    An annotated graph is a pair $(G,S)$, where $G$ is a graph and $S$ is a subset of vertices of $G$. 
    We present a unified approach for definitions of annotated versions of 
    classical graph parameters such as
    treewidth, pathwidth, treedepth, etc. 
    A generic way to incorporate a distinguished set of vertices into a graph parameter is via torsos. Given a graph $G$ and a subset $S'\subseteq V(G)$, we define $\torso_G(S')$ to be the graph with vertex set $S'$ in which distinct vertices $u,v\in S'$ are adjacent if and only if there exists a $u$--$v$ path in $G$ whose internal vertices lie outside $S'$. 
Now let $\param$ be any graph parameter. For an annotated graph $(G,S)$, we define the torso $\param$ of $(G,S)$ by
    \[\param(G,S) = \min\{\param(\torso_G(S')) \mid S \subset S' \subset V(G)\}.\]
For a minor-closed class of annotated graphs $\mathcal{C}$, we discuss equivalent conditions for $\param(G,S)$ bounded for all $(G,S)\in\mathcal{C}$ in terms of annotated graphs excluded in $\mathcal{C}$. 
In many cases, we collect known results from the literature. 
We include two new arguments for $\param$ being the $2$-treedepth and the feedback vertex number. 
Among other statements, we also include a variant of the Erdős–Pósa property for annotated cycles. 
We propose some new research directions.
\end{abstract}

\maketitle

\section{Introduction}\label{sec:introduction}

Structural graph theory provides a mathematical foundation for the development of powerful algorithmic techniques and the systematic design of efficient algorithms for computationally hard problems on structurally restricted graph classes. 
Central concepts such as graph minors and tree decompositions have played a particularly important role in this interplay between structure and algorithms. A prime example is the theory of treewidth and the celebrated Courcelle's Theorem~\cite{Courcelle90}.

In this paper, we study classical statements from structural graph theory in a setting where a graph is equipped with a distinguished set of vertices. An \defin{annotated graph} is a pair \defin{$(G,S)$}, where $G$ is a graph and $S\subseteq V(G)$. Our motivation for pursuing this line of research is twofold: algorithmic and structural.

Many problems in theoretical computer science take as input not only a graph, but also a distinguished subset of its vertices. For example, in the \textsc{Steiner Tree} problem, the input consists of a connected graph $G$ with nonnegative edge weights together with a subset of vertices of $G$, whose elements are usually referred to as \emph{terminals}. The goal is to find a connected subgraph of $G$ containing all terminals, possibly together with additional vertices, and having minimum total edge weight.

Although \textsc{Steiner Tree} is computationally hard in general, standard dynamic programming techniques solve it on $n$-vertex graphs of treewidth at most $t$ in time $f(t)\cdot n^{\mathcal{O}(1)}$. However, a tree decomposition of the input graph is oblivious to the distinguished set of terminals. In particular, the portion of the graph that is relevant for connecting the terminals may be structurally \emph{simple} and well separated from the remainder of the graph, while the remainder itself may be arbitrarily \emph{complicated}. Thus, the treewidth of the entire graph may substantially overestimate the structural complexity that is relevant to the problem.

A generic way to incorporate a distinguished set of vertices into a graph parameter is via torsos. Given a graph $G$ and a subset $S'\subseteq V(G)$, we define \defin{$\torso_G(S')$} to be the graph with vertex set $S'$ in which distinct vertices $u,v\in S'$ are adjacent if and only if there exists a \stpath{u}{v} in $G$ whose internal vertices lie outside $S'$. 
In particular, for every $S'\subseteq V(G)$, we  have  $G[S']\subseteq\torso_G(S')$.  

Now let $\param$ be any graph parameter; formally, $\param$ is a real-valued function on the isomorphism classes of finite simple graphs. For an annotated graph $(G,S)$, we define the \defin{torso $\param$} of $(G,S)$ by
    \[\defin{\param(G,S)} = \min\{\param(\torso_G(S')) \mid S \subset S' \subset V(G)\}.\]
Thus, rather than measuring the complexity of the whole graph $G$, we are allowed to enlarge the distinguished set $S$ to a set $S'$ and measure the complexity of the corresponding torso.
The simplest explanation why we should consider a superset of $S$ comes from the following example.
Consider $G$ to be a star with $S$ being the set of leaves.
Then, $\torso_G(S)$ is a complete graph while $G$ itself is very simple.

Jansen and Swennenhuis~\cite{JS24} proved that the \textsc{Steiner Tree} problem can be solved efficiently on $n$-vertex annotated graphs of torso treewidth at most $t$ in time $f(t)\cdot n^{\mathcal{O}(1)}$. 
This provides an algorithmic indication that torso versions of classical width parameters may capture the complexity of annotated instances more accurately than the corresponding parameters of the underlying graph.
The first implicit occurrence of a torso parameter that we could track is in a paper of Marx, O’Sullivan, and Razgon~\cite[Lemma 2.11]{MOR13}.

We now turn to the structural motivation for studying annotated graphs. In their seminal \emph{Graph Minors} series, spanning from 1983 to 2010, Robertson and Seymour developed a far-reaching theory describing the structure of graphs excluding a fixed graph as a minor. One of the central themes of this theory is the close relationship between excluded minors and width parameters such as treewidth and pathwidth. For instance, excluding a fixed planar graph as a minor is equivalent to having bounded treewidth, as follows from the Grid Minor Theorem~\cite{GM5}. Similarly, excluding a fixed forest as a minor is equivalent to having bounded pathwidth, as established already in the first paper of the series~\cite{GM1}.

Recently, Hodor, La, Micek, and Rambaud~\cite{quickly-excluding-apex-forest} studied analogous statements for excluded minors of apex type. Recall, for example, that a graph is an \emph{apex-forest} if it can be made acyclic by deleting at most one vertex.
They proved that for every apex-forest $X$,
$X$-minor-free graphs have bounded \emph{local} pathwidth, i.e. every radius $r$ balls induces a subgraph of pathwidth at most $\bigO(|V(X)| \cdot r)$.
Note that a similar statement with worse bounds was proved before by Dujmović, Eppstein, Joret, Morin, and Wood~\cite{lpw}.
A key novelty in the proof in~\cite{quickly-excluding-apex-forest} is the use of torso pathwidth.

These two independent developments, one arising from parameterized algorithms and the other from structural graph theory, suggest that width parameters for annotated graphs deserve a systematic study in their own right.

There is a very natural way of mimicking the minor relation in the realm of annotated graphs.
We say that a family $(B_x \mid x \in V(H))$ of  pairwise disjoint subsets of vertices of $G$ is a \defin{model} of an annotated graph $(H,T)$ in an annotated graph $(G,S)$ if
\begin{enumerate}
    \item for each $x \in V(H)$, $G[B_x]$ is a connected subgraph of $G$;
    \item for all $xy \in E(H)$, there exists $u \in B_x$ and $v \in B_y$ with $uv \in E(G)$; and
    \item for each $x \in T$, there exists $u \in S$ with $u \in B_x$.
\end{enumerate}
When there exists a model of an annotated graph $(H,T)$ in an annotated graph $(G,S)$, we say that $(H,T)$ is a \defin{minor} of $(G,S)$. In the case where $(H,V(H))$ is a minor of $(G,S)$, we also say that $H$ is an \defin{$S$-rooted minor} of $G$ and a model of $(H,V(H))$ in $(G,S)$ is called an \defin{$S$-rooted model} of $H$ in $G$.
Note that $H$ is a minor of $G$ if and only if $(H,V(H))$ is a minor of $(G,V(G))$.
Similar notions were already studied in literature~\cite{GMXIII,Kawarabayashi2021,Wollan-rooted}.

In this paper, we study the interplay between torso parameters and excluding annotated graphs as minors.
We (try to) list all the known results in this direction along with some new ones.
The graph parameters that we study are: vertex cover number, treedepth, pathwidth, treewidth, $2$-treedepth, feedback vertex number, and Hadwiger number; they are denoted respectively by: \defin{$\vc(\cdot)$}, \defin{$\td(\cdot)$}, \defin{$\pw(\cdot)$}, \defin{$\tw(\cdot)$}, \defin{$\td_2(\cdot)$}, \defin{$\fv(\cdot)$}, and \defin{$\had(\cdot)$}.
For each of the parameters, we prove or cite an approximate excluded annotated minor characterization.
We define 
    \begin{align*}
        \defin{\calF_{\vc}} &= \{(H,V(H)) \mid \text{$H$ is a matching}\},\\
        \defin{\calF_{\td}} &= \{(H,V(H)) \mid \text{$H$ is a path}\},\\
        \defin{\calF_{\pw}} &= \{(H,V(H)) \mid \text{$H$ is a forest}\},\\
        \defin{\calF_{\tw}} &= \{(H,T) \mid \text{$H$ is a plane graph and $T$ is contained in the outer face of $H$}\},\\
        \defin{\calF_{\td_2}} &= \{(H,T) \mid \text{$H$ is a ladder and $T$ is a side of $H$}\},\\
        \defin{\calF_{\fv}} &= \{(H,V(H)) \mid \text{$H$ is a disjoint union of triangles}\},\\
        \defin{\calF_{\had}} &= \{(H,V(H)) \mid \text{$H$ is a complete graph}\}.
    \end{align*}


A graph class is \defin{minor-closed} if for every graph in the class, all its minors also belong to the class.

\begin{theorem}\label{thm:main}
    For each $\param \in \{\vc, \td, \pw, \tw, \td_2, \fv, \had\}$ and for every minor-closed class of annotated graphs $\calC$, there exists a graph $(H,T) \in \calF_\param$ with $(H,T) \notin \calC$ 
    if and only if 
    there exists a constant $c$ such that for every $(G,S) \in \calC$, we have $\param(G,S) \leq c$.
\end{theorem}

See a visual summary in~\Cref{tab:table}.
\Cref{thm:main} is proved in~\Cref{sec:parameters}.
The statements for $\param \in \{\vc, \td, \pw, \tw\}$ are already known in the literature~\cite{Gallai64,quickly-excluding-apex-forest,as-quickly-as-possible,Marx2017}, and we include them for completeness and illustration without proofs.
The statement for $\param = \had$ can also be derived from known results, see~\cite[Corollary 13]{GMT-revisited} or~\cite[Theorem~1.2]{ColorfulMinors}.
Since, to our knowledge, there is no explicit proof in the literature, we include the proof.
Finally, the proofs for $\param \in \{\td_2,\fv\}$ are new. 
It is well-known that paths are obstructing bounded treedepth, however, the treedepth of a path is logarithmic in its size.
Kawarabayashi and Rossman~\cite{Kawarabayashi2021} found obstructions for bounded treedepth that have polynomial treedepth.
We provide an annotated variant of this statement, see~\Cref{th:focusedABC}.


\begin{table}[tp]
    \centering
    \renewcommand{\arraystretch}{2.5}
    \setlength{\tabcolsep}{20pt}
    \tikzset{every picture/.style={every node/.style={draw, circle, fill, minimum size = 4pt, inner sep = 0pt}}}
    \begin{tabular}{c | c | c}
         $\param$ & $\mathcal{F}_\param$ & Reference \\
         \hline
         $\vc$ & 
             \tikz[scale=0.5, baseline=-3pt]{\foreach \i in {0,...,4}{
                 \node[rectangle, red] (u\i) at (\i,-0.5) {};
                 \node[rectangle, red] (v\i) at (\i,0.5) {};
                 \draw (u\i) -- (v\i);
                 \draw[opacity=0] (-0.5,-1) rectangle (4.5, 1);
             }} & \Cref{thm:vc} \\
        $\td$ & 
             \tikz[scale=0.5, baseline=-3pt]{
               \foreach \i in {0,...,5}{
                 \node[rectangle, red] (u\i) at (\i,0) {};
                }
                \foreach \i [evaluate=\i as \j using {int(\i+1)}] in {0,...,4}{
                    \draw (u\i) -- (u\j);
                }
                \draw[opacity=0] (-0.5,-1) rectangle (5.5, 1);
             } & \cite[Theorem~6]{quickly-excluding-apex-forest} \\
        $\pw$ & 
             \tikz[xscale=0.5, yscale=0.33, baseline=(v0.base)]{
                \node[rectangle, red] (v) at (0,0) {};
                
                \node[rectangle, red] (v0) at (-1,-0.87) {}; \draw (v) -- (v0);
                \node[rectangle, red] (v00) at (-1.5,-1.7) {}; \draw (v00) -- (v0);
                \node[rectangle, red] (v01) at (-0.5,-1.7) {}; \draw (v01) -- (v0);
                
                \node[rectangle, red] (v1) at (1,-0.87) {}; \draw (v1) -- (v);
                \node[rectangle, red] (v10) at (1.5,-1.7) {}; \draw (v10) -- (v1);
                \node[rectangle, red] (v11) at (0.5,-1.7) {}; \draw (v11) -- (v1);

                \draw[opacity=0] (-2,-3) rectangle (2,0.5);
             } & \cite[Theorem~7]{quickly-excluding-apex-forest} \\
        $\tw$ & 
             \tikz[scale=0.33, baseline=(v22.base)]{ 
                \node[rectangle, red] (v00) at (0,0) {};
                \draw (v00) -- (0,1);
                \draw (v00) -- (1,0);
                
                 \foreach \x in {1,...,3}{
                     \node[rectangle, red] (v\x0) at (\x,0) {};
                     \draw (v\x0) -- (\x,1);
                     \draw (v\x0) -- (\x+1,0);
                     \foreach \y in {1,...,3}{
                         \node (v\x\y) at (\x,\y) {};
                         \draw (v\x\y) -- (\x, \y+1);
                         \draw (v\x\y) -- (\x+1, \y);
                     }
                 }
                 
                \foreach \y in {1,...,3}{
                     \node[rectangle, red] (v0\y) at (0,\y) {};
                     \draw (v0\y) -- (0, \y+1);
                     \draw (v0\y) -- (1, \y);
                }
                \foreach \y in {0,...,3}{                     
                     \node[rectangle, red] (v4\y) at (4,\y) {};
                     \draw (v4\y) -- (4, \y+1);
                 }

                 \foreach \x in {0,...,3}{
                     \node[rectangle, red] (v\x4) at (\x,4) {};
                     \draw (v\x4) -- (\x+1,4);
                 }
                 \node [rectangle, red] at (4,4) {};

                 \draw[opacity=0] (-1,-1) rectangle (5,5);
             } & \cite{Marx2017} and \cite[Theorem~8]{quickly-excluding-apex-forest} \\
        $\td_2$ & 
             \tikz[scale=0.4, baseline=3pt]{ 

                 \foreach \x in {0,...,5}{
                     \node[rectangle, red] (u\x) at (\x,0) {};
                     \node (v\x) at (\x,1) {};
                     \draw (u\x) -- (v\x);
                 }
                 \foreach \x [evaluate=\x as \xx using {int(\x+1)}] in {0,...,4}{
                      \draw (u\x) -- (u\xx);
                      \draw (v\x) -- (v\xx);
                 }
                 \draw[opacity=0] (-0.5,-1) rectangle (5.5,2);
             } & \Cref{thm:rooted:tdtwo:ladder} \\
        $\fv$ & 
             \tikz[scale=0.4, baseline=0pt]{ 
                 \foreach \x in {0,...,2}{
                     \begin{scope}[shift={(2.66*\x,0)}, rotate=-30]
                         \node[rectangle, red] (u\x0) at (0:0.8) {};
                         \node[rectangle, red] (u\x1) at (120:0.8) {};
                         \node[rectangle, red] (u\x2) at (240:0.8) {};
                         \draw (u\x0) -- (u\x1) -- (u\x2) -- (u\x0);
                     \end{scope}
                 }
                 \draw[opacity=0] (-1,-1) rectangle (5,2);
             } & \Cref{thm:fv_EP} \\
        $\had$ & 
             \tikz[scale=0.5, baseline=-3pt]{ 
                \node[rectangle, red] (v0) at (0:1) {};
                 \foreach \x [evaluate=\x as \xx using {int(\x-1)}] in {1,...,4}{
                     \node[rectangle, red] (v\x) at (\x*360/5:1) {};
                     \foreach \y in {0,...,\xx}{
                         \draw (v\y) -- (v\x);
                     }
                 }
                 \draw[opacity=0] (-2,-2) rectangle (2,2);
             } & \Cref{thm:had} \\
    \end{tabular}

    \vspace{10pt} 
    
    \caption{Summary of \Cref{thm:main}, 
        which assert that, for each parameter
        $\param \in \{\vc,\td,\pw,\tw,\td_2,\fv,\had\}$, 
        a minor-closed class of annotated graphs has bounded $\param$ if and only if 
        it excludes a member of $\mathcal{F}_\param$.
        Here, the annotated vertices are depicted as red squares.}
    \label{tab:table}
\end{table}

\subsection{Erdős–Pósa property for annotated models}
In 1965, Erdős and Pósa~\cite{EP65} proved that there is a function $f(k) = \Oh(k\log k)$ such that for every graph $G$ and every positive integer $k$, either $G$ contains $k$ vertex-disjoint cycles, or there is a set $X$ of at most $f(k)$ vertices of $G$ such that $G-X$ is a forest.
Informally, we say that cycles have \emph{the Erdős–Pósa property}.
Here, we are interested in Erdős–Pósa property of models of a fixed graph $H$.
It is well-known that models of $H$ have the Erdős–Pósa property if and only if $H$ is planar~\cite{GM5}.
The positive direction of the statement above follows from the grid minor theorem.
Van Batenburg, Huynh, Joret, and Raymond~\cite{tightEP} proved that when $H$ is planar, the function witnessing the Erdős–Pósa property of models of $H$ is in $\Oh(k\log k)$.

We are interested in the Erdős--Pósa property of models of annotated graphs in annotated graphs.
The annotated variant of the grid theorem (\Cref{thm:main} for $\param = \tw$ or \cite[Theorem~9]{quickly-excluding-apex-forest}) implies that if $H$ is a connected planar graph with a fixed plane embedding and $T$ is a subset of the vertices of the outer face of $H$, then the models of $(H,T)$ in annotated graphs have the Erdős--Pósa property.

\begin{theorem}[{\cite[Corollary 24]{quickly-excluding-apex-forest}}]\label{EP-sota}
    For every connected graph $H$ with a fixed plane embedding,
    for every subset $T \subset V(H)$ of the vertices on the outer face of $H$,
    there exists a function $f_{(H,T)}\colon \mathbb{N} \to \mathbb{N}$ such that
    for every annotated graph $(G,S)$, and
    for every positive integer $k$, 
    either
    \begin{enumerate}
        \item $(G,S)$ has $k$ disjoint models of $(H,T)$ or\label{item:EP_thm_i}
        \item there exists a set $Z \subseteq V(G)$ such that $|Z| \leq f_{(H,T)}(k)$ and $G-Z$ has no models of $(H,T)$.\label{item:EP_thm_ii}
    \end{enumerate}
\end{theorem}

Combining with the most recent improvements on the grid theorem due to Chuzhoy and Tan~\cite{Chuzhoy2021} and Bonnet~\cite{Bonnet}, we may have $f_{(H,T)}(k) = \widetilde{\Oh}(k^{33})$ in~\Cref{EP-sota}.
We conjecture that $f_{(H,T)}(k) = \Oh(k\log k)$ is possible, see \cref{conj:rooted_EP_k_logk}.
As evidence, we prove the first non-trivial case.

\begin{theorem}\label{thm:EP-triangles-in-one-vx}
    Let $H = K_3$ and let $T$ be a one-vertex subset of $V(H)$.
    There exists a constant $c$ such that 
    for every annotated graph $(G,S)$ and
    for every positive integer $k$, 
    either
    \begin{enumerate}
        \item $(G,S)$ has $k$ disjoint models of $(H,T)$ or
        \item there exists a set $Z \subseteq V(G)$ such that $|Z| \leq c \cdot  k \log k$ and $(G-Z, S \setminus Z)$ has no model of $(H,T)$.
    \end{enumerate}
\end{theorem}

\Cref{thm:EP-triangles-in-one-vx} follows from~\Cref{thm:EP-lollipops} proved in~\Cref{sec:EP}.
Note that when $H$ is a forest, then the pathwidth part of~\Cref{thm:main}, namely \cref{thm:pw}, implies that one may take $f_{(H,T)}$ quadratic (see~\cite[Corollary~25]{quickly-excluding-apex-forest}).
Recently Claus, Joret, Rambaud, and Robinson~\cite{EProotedtrees} showed that in this case, $f_{(H,T)}$ can be linear.

\subsection{Minor-closed classes of annotated graphs}

One of the fundamental results of structural graph theory due to Robertson and Seymour is that all graphs are well-quasi ordered under the minor relation.
In other words, every minor-closed class of graphs has a finite list of minor-minimal obstructions.
In fact, the result of Robertson and Seymour~\cite{GMXXIII} is general enough so it can be also applied to annotated graphs, see also~\cite[Section 2.4]{ColorfulMinors}.
Following the approach that we took for graph parameters, every minor-closed class of graphs yields a minor-closed class of annotated graphs.
For every class of graphs $\calC$, let
\[
    \defin{\calC^\torsobullet} = \big\{(G,S) \ \big\vert\ \text{there exists $S'\subset V(G)$ with $S \subset S'$ such that $\torso_G(S') \in \calC$}\big\}.
\]
The following meta-question arises.
Given a class of graphs $\calC$ for which we know the minor-minimal obstructions, what are the minor-minimal obstructions of $\calC^\torsobullet$?
We solve the first nontrivial case where $\calC$ is the class of forests.

\begin{restatable}{theorem}{thmforests}\label{thm:obstructions-forests}
    Let $\calC$ be the class of all forests.
    For every annotated graph $(G,S)$, we have $(G,S) \in \calC^\torsobullet$ if and only if $K_3$ is not an $S$-rooted minor of $G$.
\end{restatable}

\Cref{thm:obstructions-forests} is proved in~\Cref{ssec:fv}.
It would be interesting to establish a similar theorem for $\calC$, being all the planar graphs, i.e.\ an \emph{annotated Kuratowski theorem}, see \cref{annotated-Kuratowski}.

\subsection{Other related results}
Annotated graphs are of neuralgic importance in recent results in the Graph Minors literature. 
The central notion behind several of these developments is the one of \emph{bidimensionality} of a vertex set $S$ in a graph $G$, i.e.\ the maximum integer $\ell$ such that the $\ell\times\ell$ grid is contained as an $S$-rooted minor in $G$.
Bidimensionality was used to rephrase the graph minors structure theorem~\cite{Bidimensionality,GorskyPS26} but also exhibits the optimality of the improved bounds for the unique linkage theorem~\cite{CavallaroGKTWOptimalBounds}.
A variant of bidimensionality provides the framework for showing a dichotomy result for the applicability of the irrelevant vertex technique~\cite{ThilikosW26irr}. 
Bidimensionality has also been studied from the lens of logic, enabling algorithmic meta-theorems for corresponding extensions of first-order logic~\cite{SST25,SauSSSTV26}.

Annotated graphs can be seen as a special instance of the more general notion of \emph{colorful graphs}, graphs equipped with a fixed number of subsets of vertices, which were introduced by Protopapas, Thilikos, and Wiederrecht~\cite{ColorfulMinors}. 
Furthermore, they~\cite{PTW26} characterized which models of colorful graphs admit Erdős–Pósa property.
See also related work of Kwon and Marx~\cite{EP-rooted}.







\section{Preliminaries}\label{sec:preliminaries}


For each positive integer $k$, we write $\defin{\text{$[k]$}}=\{1,\ldots,k\}$, and $\defin{\text{$[0]$}} = \emptyset$.
All graphs considered in this paper are finite, simple, and undirected. 
The \defin{null graph} is the graph with no vertices.  
Connected graphs are nonnull.
Trees are defined as connected forests, thus trees and subtrees are also assumed to be nonnull.
Given a graph $G$, we denote by \defin{$V(G)$} its vertex set and by \defin{$E(G)$} its edge set.
A \defin{graph parameter} is a function mapping each graph to a real number.
A graph parameter $\param$ is \defin{minor-monotone} if for all graphs $H$ and $G$, the fact that $H$ is a minor of $G$ implies $\param(H) \leq \param(G)$.
Two subgraphs of a graph are \defin{disjoint} if they are vertex-disjoint. 

Let $G$ and $H$ be graphs. We denote by \defin{$G\cup H$} (resp. \defin{$G\cap H$}) the graph with vertex set $V(G) \cup V(H)$ (resp. $V(G)\cap V(H)$) and edge set $E(G)\cup E(H)$ (resp. $E(G)\cap E(H)$).
Given two subgraphs $A$ and $B$ of $G$, we say that $(A,B)$ is a \defin{separation} of $G$ if $A\cup B= G$.
We call $|V(A)\cap V(B)|$ the \defin{order} of the separation $(A,B)$.
Given $X\subseteq V(G)$, we use \defin{$G - X$} to denote the graph obtained from $G$ by removing the vertices of $X$ and all their incident edges.
For a collection $\calH$ of subgraphs of a graph $G$, we use \defin{$V(\mathcal{H})$} to denote the set $\bigcup\{V(H) \mid H\in\calH\}$.
We say that $H$ is a \defin{subdivision} of $G$ if it can be obtained from $G$ by replacing each edge by a path of length\footnote{The \defin{length} of a path is the number of its edges.} at least one.

For all $S,T \subset V(G)$, an \defin{\stpath{S}{T}} is a path in $G$ that is either 
a one-vertex path with a vertex in $S\cap T$, or 
a path with one endpoint in $S \setminus T$ and the other endpoint in $T\setminus S$ such that no internal vertices are in $S \cup T$. 
When $S$ or $T$ is a singleton, say $S=\{u\}$, we simplify the notation to \defin{\stpath{u}{T}} and analogous variants.


We state and prove several simple observations.

\begin{obs}\label{obs:projecting-S-rooted-minors-in-torso}
    Let $(G,S)$ be an annotated graph, let $H$ be a graph, and let $S'\subset V(G)$ with $S \subset S'$.
    If $H$ is an $S$-rooted minor of $G$, then $H$ is a minor of $\torso_G(S')$.
\end{obs}
\begin{proof}
    Suppose that $\big( B_x \mid x \in V(H)\big)$ is an $S$-rooted model of $H$ in $(G,S)$.
    For every $x \in V(H)$, let $B_x' = B_x \cap S'$.
    We claim that $\big( B_x' \mid x \in V(H)\big)$ is a model of $H$ in $G' = \torso_G(S')$.
    Let $x,y \in V(H)$.
    First, $B'_x \neq \emptyset$ as $B_x \cap S \neq \emptyset$. 
    Next, $G'[B'_x]$ is connected since $G[B_x]$ is connected and for every $u,v\in B'_x$ the restriction to $G'[B'_x]$ of a $u$-$v$ path in $G[B_x]$, is a $u$-$v$ path in $G'[B'_x]$. 
    Finally, if $B_x$ and $B_y$ are adjacent in $G$, for $u\in B'_x$ and $v\in B'_y$ there is a $u$-$v$ path in $G[B_x \cup B_y]$, hence $B'_x$ and $B'_y$ are adjacent in $G'$.
\end{proof}

The key implication of \cref{obs:projecting-S-rooted-minors-in-torso} is the following statement.

\begin{obs}\label{obs:transfer-large-value}
    Let $\param$ be a minor-monotone graph parameter.
    Let $H$ be a graph and $n$ be an integer such that $\param(H) \geq n$.
    Then, for every annotated graph $(G,S)$ such that $G$ contains an $S$-rooted minor of $H$, we have $\param(G,S) \geq n$.
\end{obs}

\begin{obs}\label{obs:monotonicity}
    Let $\param$ be a minor-monotone graph parameter.
    Then, for all annotated graphs $(H,T)$ and $(G,S)$, if $(H,T)$ is a minor of $(G,S)$, then $\param(H,T) \leq \param(G,S)$.
\end{obs}
\begin{proof}
Let $(B_x\mid x\in V(H))$ be a model of $(H,T)$ in $(G,S)$.
Choose $S\subseteq S'\subseteq V(G)$ such that
$\param(\operatorname{torso}_G(S'))=\param(G,S)$, and put $T'=\{x\in V(H):B_x\cap S'\ne\emptyset\}$.
Since $S\subseteq S'$, we have $T\subseteq T'$.
The sets $(B_x\cap S'\mid x\in T')$ form a model of
$\torso_H(T')$ in $\torso_G(S')$.
This can be proved with a similar argument as in  \cref{obs:projecting-S-rooted-minors-in-torso}.
Therefore, by the minor-monotonicity of $\param$,
\[
    \param(H,T)
    \le \param(\operatorname{torso}_H(T'))
    \le \param(\operatorname{torso}_G(S'))
    =\param(G,S). \qedhere
\]
\end{proof}


\section{Annotated obstructions for torso parameters}\label{sec:parameters}


In this section, we prove~\Cref{thm:main} and some other complementary results.
Given a graph parameter $\param$, 
to obtain a statement as in~\Cref{thm:main}, one needs two directions.
First, one must prove that $\param$ can be arbitrarily large on graphs in $\calF_\param$ and second, one must prove that if $\param(G,S)$ is sufficiently large, then $(G,S)$ contains as a minor a \q{large} member of $\calF_p$.
For some of the parameters, we provide an equivalent definition, which is supposed to be more comprehensive than the abstract torso definition.

\subsection{Treewidth}

A \defin{tree decomposition} of a graph $G$ is a pair $\mathcal{T} = \big(T,(W_x \mid x \in V(T))\big)$,
where $T$ is a tree and $W_x \subseteq V(G)$ for every $x \in V(T)$, satisfying the following conditions:
\begin{enumerate}
    \item for every $u \in V(G)$, the graph $T[\{x \in V(T) \mid u \in W_x\}]$ is a connected subgraph of $T$ and
    \item for every edge $uv \in E(G)$, there exists $x \in V(T)$ such that $u,v \in W_x$.
\end{enumerate}
The sets $W_x$ are called the \defin{bags} of $\mathcal{T}$.
The \defin{width} of $\mathcal{T}$ is $\max\{|W_x|-1 \mid x\in V(T)\}$.
The \defin{treewidth} of a graph $G$, denoted by \defin{$\tw(G)$}, is the minimum width of a tree decomposition of $G$.

Since for any graph $G$ and $X\subseteq V(G)$, 
the neighborhood in $X$ of a component of $G-X$ is a clique in $\torso_G(X)$,
we obtain the following observation.
\begin{obs}\label{obs:tw-decomposition}
    For every annotated graph $(G,S)$ and for every positive integer $k$, we have $\tw(G,S) \leq k$ if and only if there exists a set $S' \subset V(G)$ with $S \subset S'$ such that there exists a tree decomposition $\big(T,(W_x \mid x \in V(T))\big)$ of $G[S']$ of width at most $k$ such that for every component of $G - S'$ there exists $x \in V(T)$ with $W_x$ containing the neighborhood of the component in $G$.
\end{obs}

We call a tree decomposition $\big(T, (W_x \mid x \in V(T))\big)$
as in the observation a tree decomposition of $(G,S)$.

Let \defin{$\gm$} be the minimum function such that for every positive integer $\ell$, if a graph $G$ has no model of the $\ell \times \ell$ grid, then $\tw(G) \leq \gm(\ell)$.
This function exists due to the Grid Minor Theorem, and the best known upper bound is $\gm(\ell) =\widetilde{\Oh}(\ell^{8})$ due to Bonnet~\cite{Bonnet} and Chuzhoy and Tan~\cite{Chuzhoy2021}. The following result was given by Hodor, La, Micek, and Rambaud~\cite{quickly-excluding-apex-forest}.

\begin{theorem}[{\cite[Theorem~8\footnote{Theorem~8 in the arXiv version, renumbered Theorem~3.3 in the published version.}]{quickly-excluding-apex-forest}}]\label{thm:rooted:gridminor}
    There is a function $f_{\ref{thm:rooted:gridminor}}\colon\mathbb{N}\to\mathbb{N}$ such that the following holds.
    Let $H$ be a plane graph and let $T$ be a subset of the outer face of $H$.
    For every annotated graph $(G,S)$, if $\tw(G,S) \geq f_{\ref{thm:rooted:gridminor}}(|V(H)|)$, then $(G,S)$ contains a model of $(H,T)$. Moreover, we can take $f_{\ref{thm:rooted:gridminor}}(x)= 3\gm(98304 \cdot x^4)+2$.
    %
\end{theorem}


\Cref{thm:rooted:gridminor} is proved via tangles of annotated graphs.
The main technical part of the proof was stated implicitly by Marx, Seymour, and Wollan \cite[statement~1.3]{Marx2017}.
A complete proof with all the necessary adjustments to the annotated setting can be found in the appendix of~\cite{quickly-excluding-apex-forest}.

In the further parts of the paper, we will use the following Helly-like property of annotated graphs of bounded torso treewidth.
The statement is an extension of \cite[Lemma~23]{quickly-excluding-apex-forest} that follows directly from the proof.
Given a tree decomposition $\calT = \big(T,(W_x \mid x \in V(T))\big)$ of 
an annotated graph $(G,S)$ and a subgraph $H$ of $G$ intersecting $S$, 
we define $\defin{\calT[H]} = \{x \in V(T) \mid W_x \cap V(H) \neq \emptyset\}$.
By the properties of tree decompositions, if $H$ is connected, 
then $\mathcal{T}[H]$ induces a connected subtree of $T$.

\begin{lemma}[Extension of {\cite[Lemma 23]{quickly-excluding-apex-forest}}]\label{lem:tw:erdos:posa}
  Let $(G,S)$ be an annotated graph, let $S' \subset V(G)$ with $S \subset S'$, let $\big(T,(W_x \mid x \in V(T))\big)$ be a tree decomposition of $(G,S')$, and let $\calH$ be a family of connected subgraphs of $G$ each containing a vertex of $S$. For every positive integer $k$, either
  \begin{enumerate}
        \item there are $k$ subgraphs in $\calH$ so that each two $H$ and $H'$ satisfy $\calT[H] \cap \calT[H'] = \emptyset$ or 
        \item there is a set $X \subset V(T)$ with $|X| \leq k-1$ such that if $Z = \bigcup_{x \in X} W_x$, then $V(H) \cap Z \neq \emptyset$ for every $H \in \calH$.
    \end{enumerate}
\end{lemma}

A consequence of \cref{lem:tw:erdos:posa} is the following:
let $G$ be the $n \times n$ grid, and let $S$ be the vertex set of the first row of this grid.
If $\mathcal{H}$ consists of all the subgraphs of $G$ induced by
the union of a row and a column of $G$,
then \cref{lem:tw:erdos:posa} for $k=2$ implies that
there is a set $Z$ of $\tw(G,S)+1$ vertices intersecting
every member of $ \mathcal{H}$,
and we deduce $\tw(G,S) \geq n-1$.

This and \Cref{thm:rooted:gridminor} yields the treewidth statement in \cref{thm:main}.

\subsection{Pathwidth}
A \defin{path decomposition} of a graph $G$ is a tree decomposition 
$\big(T,(W_x\mid x\in V(T))\big)$ of $G$ where $T$ is a path.
The \defin{pathwidth} of a graph $G$, denoted by \defin{$\pw(G)$},
is the minimum width of a path decomposition of $G$.

The following observation is the analog of
\cref{obs:tw-decomposition} for path decompositions.

\begin{obs}
    For every annotated graph $(G,S)$ and for every positive integer $k$, we have $\pw(G,S) \leq k$ if and only if there exists a set $S' \subset V(G)$ with $S \subset S'$ such that there exists a path decomposition $\big(W_i \mid i \in [\ell]\big)$ of $G[S']$ of width at most $k$ such that for every component of $G - S'$ there exists $i \in [\ell]$ with $W_i$ containing the neighborhood of the component in $G$.
\end{obs}

The following theorem and \cref{obs:transfer-large-value} yield the pathwidth statement in \cref{thm:main}.

\begin{theorem}[{\cite[Theorem~7]{quickly-excluding-apex-forest}}] \label{thm:pw}
    For every annotated graph $(G,S)$ and for every positive integer $k$, 
    if $\pw(G,S) \geq 2k-1$, then 
    $(G,S)$ contains an $S$-rooted model of every forest on $k$ vertices. 
\end{theorem}

\subsection{Vertex cover number} \label{ssec:vcn}
The \defin{vertex cover number} of a graph $G$, denoted by \defin{$\vc(G)$}, is the minimum number $k\in\mathbb{N}$ for which there is a set $X\subseteq V(G)$ with $|X|=k$ such that every edge of $G$ has at least one endpoint in $X$.
Such a set $X$ is called a \defin{vertex cover} in $G$.

\begin{obs}\label{obs:vc}
    For every annotated graph $(G,S)$ and for every positive integer $k$, 
    we have $\vc(G,S) \leq k$,
    if and only if there exists a set $X \subset V(G)$ with $|X| \leq k$ 
    such that there is no $(S \setminus X)$-rooted model of $K_2$ in $G-X$.
\end{obs}
\begin{proof}
    Let $(G,S)$ be an annotated graph and let $k$ be a positive integer.
    First, assume that $\vc(G,S) \leq k$.
    Consider a vertex subset $S'\subset V(G)$ with $S\subseteq S' $ and $\vc(\torso_G(S')) = \vc(G,S)$ 
    Let $X \subset S'$ with $|X| \leq k$ be such that $\torso_G(S') - X$ is edgeless.
    Suppose that there is an $(S\setminus X)$-rooted minor of $K_2$ in $G-X$.
    It follows that there is an \stpath{s_1}{s_2} between distinct vertices $s_1,s_2 \in S$ in $G - X$ and since $s_1,s_2 \in S'$, there is an \stpath{s_1}{s_2} in $\torso_G(S')-X$, which is a contradiction.
    
    Next, suppose that there exists a set $X \subset V(G)$ with $|X| \leq k$ such that there is no $(S \setminus X)$-rooted model of $K_2$ in $G-X$.
    We claim that $\torso_G(S \cup X) - X$ is edgeless.
    Otherwise, there are two distinct $s_1,s_2 \in S\setminus X$ connected by an edge in $\torso_G(S \cup X) - X$.
    However, this implies that there is an \stpath{s_1}{s_2} in $G - X$, hence an $(S \setminus X)$-rooted model of $K_2$, which is again a contradiction.
\end{proof}

    

Let $(G,A)$ be an annotated graph.
An \defin{$A$-path} in $G$ is a path with at least two vertices and both ends in $A$. 
In 1964, Gallai~\cite{Gallai64} proved that either $G$ contains $k$ disjoint $A$-paths, or there exists a subset $X$ of the vertices of $G$ with $|X|\leq 2k-2$ such that every $A$-path in $G$ contains a vertex of $X$.
Note that the latter statement (by~\Cref{obs:vc}) is equivalent to $\vc(G,A) \leq 2k-2$.
Therefore, we immediately obtain one of the implications of the vertex cover statement in \cref{thm:main}.
The other implication follows from \cref{obs:transfer-large-value}.

\begin{theorem}\label{thm:vc}
    For every annotated graph $(G,S)$ and for every positive integer $k$, if $\vc(G,S) \geq 2k-1$, then $G$ contains an $S$-rooted model of the perfect matching on $2k$ vertices.
\end{theorem}

\subsection{Feedback vertex number}\label{ssec:fv}

The \defin{feedback vertex number} of a graph $G$, denoted by \defin{$\fv(G)$}, is the minimum number $k\in\mathbb{N}$ for which there is a set $X\subseteq V(G)$ with $|X|=k$ such that $G-X$ is acyclic.

We begin this subsection by proving~\Cref{thm:obstructions-forests}, which we restate for convenience

\thmforests*
\begin{proof}
    First, note that if $(K_3,V(K_3))$ is a minor of an annotated graph $(G,S)$, then by~\Cref{obs:projecting-S-rooted-minors-in-torso}, $K_3$ is a minor of $\torso_G(S')$ for every $S' \subset V(G)$ with $S \subset S'$,
    and so $(G,S) \notin \calC^\torsobullet$.
    
    Therefore, to conclude the proof, it suffices to show that for every annotated graph $(G,S)$, if $(K_3,V(K_3))$ is not a minor of $(G,S)$, then $(G,S) \in \calC^\torsobullet$.
    Let $(G,S)$ be an annotated graph with no $(K_3,V(K_3))$ minor.
    By considering each component separately, we may assume without loss of generality that $G$ is connected.
    When $S = \emptyset$, then $(G,S) \in \calC^\torsobullet$, as witnessed by $S' = \emptyset$.
    Thus, we may assume that $S \neq \emptyset$.
    Let $T$ be a connected subgraph of $G$ containing $S$ and with a minimum number of edges under these constraints.
    Observe that $T$ is a tree.
    Let $B \subset V(T)$ be the set of vertices of degree at least $3$ in $T$.
    We define
    \[S' = S \cup B.\]
    We shall prove that $\torso_G(S')$ is a forest witnessing that $(G,S) \in \calC^\torsobullet$.
    To do so, we first prove the following:
    \begin{claim} \label{cla:fv_distinct_comp}
        For every vertex $u \in S'$,
        for every pair of distinct components $T_1,T_2$ of $T-u$,
        the sets $V(T_1)$ and $V(T_2)$ lie in distinct components of $G-u$.
    \end{claim}

    \begin{proofclaim}
        Let $u \in S'$.
        Suppose for contradiction that there is a path $P$ in $G-u$
        between $V(T_1)$ and $V(T_2)$  for two components $T_1$ and $T_2$ of $T-u$.
        By possibly replacing $T_2$ by the first component of $T-u$ intersecting
        $P$ after $T_1$, 
        we assume that $T_1$ and $T_2$ are the only components of $T-u$
        intersecting $P$.
        We denote by $U$ the set of all the internal vertices of $P$.
        If $u \in B$ (i.e. $u$ has degree more than $2$ in $T$),
        then there is a component $T_3$ of $T-u$ distinct from $T_1$ and $T_2$,
        and we set $A = \{u\} \cup V(T_3)$.
        Note that by the minimality of $T$, $T_3$ contains a vertex in $S$.
        Otherwise, $u$ is in $S$, and we set $A = \{u\}$.
        In both cases, 
        $A$ is a set of vertices that induces a connected subgraph of $G$, 
        and that contains both $u$ and a vertex of $S$.
        Consider $(V(T_1) \cup U, V(T_2), A)$.
        By construction, it is a model of $K_3$.
        Moreover, by minimality of $E(T)$, 
        $V(T_1)$ and $V(T_2)$ both intersect $S$.
        Since $A$ also intersects $S$, this model is $S$-rooted.
        This contradiction proves the claim.
    \end{proofclaim}

    Now, let $E' = \{uv \mid u,v\in S', \text{vertices between $u$ and $v$ in $T$ are not in $S'$}\}$, 
    and observe that $T$ is a subdivision of $(S',E')$. 
    In particular, $(S',E')$ is a tree.
    We claim that $(S',E')$ is exactly $\torso_G(S')$. 
    By definition, $E'\subseteq E(\torso_G(S'))$. 
    Suppose now that there is $uv \in E(\torso_G(S'))$ such that $uv \notin E'$. 
    Then there is $w\in S'$ between $u$ and $v$ in $T$.
    By \cref{cla:fv_distinct_comp}, the components $T_u$ and $T_v$ of $T-w$ 
    containing $u$ and $v$ respectively lie in distinct components of $G-w$. 
    Therefore, there is no edge between $u$ and $v$ in $\torso_G(S')$, 
    a contradiction.
    This proves that $\torso_G(S') = (S',E')$, and so is a forest.
    %
    %
\end{proof}


\Cref{thm:obstructions-forests} implies the following description of annotated graphs $(G,S)$ with $\fv(G,S) \leq k$.

\begin{corollary}\label{cor:fv}
    For every annotated graph $(G,S)$ and for every positive integer $k$, 
    we have $\fv(G,S) \leq k$,
    if and only if there exists a set $X \subset V(G)$ with $|X| \leq k$ 
    such that there is no $(S \setminus X)$-rooted model of $K_3$ in $G-X$.
\end{corollary}
\begin{proof}
    Let $(G,S)$ be an annotated graph and $k$ be a positive integer.
    First, assume that $\fv(G,S) \leq k$.
    Let $S' \subset V(G)$ with $S \subset S'$ be such that $\fv(\torso_G(S')) = \fv(G,S)$
    and let $X \subset S'$ with $|X| \leq k$ be such that $\torso_G(S') - X$ is a forest.
    Then, by \cref{thm:obstructions-forests}, $G - X$ has no $(S\setminus X)$-rooted model of $K_3$.
    
    Otherwise, assume there is $X \subset V(G)$ with $|X| \leq k$ such that there is no $(S \setminus X)$-rooted model of $K_3$ in $G-X$. 
    By \cref{thm:obstructions-forests}, $(G-X, S\setminus X) \in \calC^\torsobullet$.
    Hence, there is $S'\subseteq V(G)$ with $(S\setminus X)\subseteq S'$ 
    such that $\torso_{G-X}(S')$ is a forest,
    and so $\fv(\torso_G(S' \cup X)) \leq |X| \leq k$.
\end{proof}

We next show that if $\fv(G,S)$ is large, then $G$ contains a large collection of disjoint $K_3$'s as $S$-rooted minors. The reverse implication is an immediate corollary of~\cref{cor:fv} (combined with~\cref{obs:transfer-large-value}). This way, we obtain the feedback vertex number statement in~\cref{thm:main}.

\begin{theorem} \label{thm:fv_EP}
    There exists a function $f_{\ref{thm:fv_EP}}$ such that for every annotated graph $(G,S)$ and for every positive integer $k$, if $\fv(G,S) \geq f_{\ref{thm:fv_EP}}(k)$, then $(G,S)$ contains an $S$-rooted model of $H_k$, where $H_k$ is $k$ disjoint copies of $K_3$. 
\end{theorem}
\begin{proof}
    Let $k$ be a positive integer and let $g(k) = f_{\ref{thm:rooted:gridminor}}(|V(H_k)|)$ be the value given by~\Cref{thm:rooted:gridminor} such that for every annotated graph $(G,S)$ if $\tw(G,S) \geq g(k)$, then $(G,S)$ contains an $S$-rooted model of $H_k$.
    Note that $H_k$ is a planar graph admitting a plane embedding where all the vertices are in the outer face.
    We define $f_{\ref{thm:fv_EP}}(k) = (k-1)g(k)+1$.
    
    Let $(G,S)$ be an annotated graph with $\fv(G,S) \geq f_{\ref{thm:fv_EP}}(k)$.
    If $\tw(G,S) \geq g(k)$, then by~\Cref{thm:rooted:gridminor}, $(G,S)$ contains an $S$-rooted model of $H_k$, as desired.
    Thus, we assume that $\tw(G,S) < g(k)$.
    Let $\calH$ be the family of the unions of branch sets of $S$-rooted models of $K_3$ in $(G,S)$.
    By~\Cref{lem:tw:erdos:posa} applied to a tree decomposition of width $\tw(G,S)$,
    either there are $k$ disjoint members of $\mathcal{H}$ and so $H_k$ is an $S$-rooted minor of $G$, or
    there exists a set $X \subset V(G)$ with $|X| \leq (k-1) g(k)$ such that every subgraph in $\calH$ contains a vertex of $X$.
    Suppose we are in the latter case.
    In particular, $(G - X,S \setminus X)$ has no $S\setminus X$-rooted model of $K_3$.
    Therefore, 
    \Cref{cor:fv} implies that $\fv(G,S) \leq |X| < f_{\ref{thm:fv_EP}}(k)$.
\end{proof}

\subsection{Treedepth}

A \defin{rooted forest} is a disjoint union of rooted trees.
The \defin{vertex-height} of a rooted forest $F$ is the
maximum number of vertices on a path from a root to a leaf in $F$, and the \defin{depth} of a vertex $u\in V(F)$ is the number of vertices in the path between $u$ and the root of its component.
For two vertices $u,v$ in a rooted forest $F$, we say that $u$ is a \defin{descendant} of $v$ in $F$ if $v$ lies on the path from a root to $u$ in $F$. The \defin{closure} of $F$ is the graph with vertex set $V(F)$ and edge set
$\{vw \mid v\neq w \text{ and $v$ is a descendant of $w$ in $F$}\}$.
We say that $F$ is an \defin{elimination forest} of $G$ if $V(F) = V(G)$ and $G$ is a subgraph of the closure of $F$. The \defin{treedepth} of a graph $G$, denoted by \defin{$\td(G)$}, is $0$ if $V(G)$ is empty, and otherwise is the minimum vertex-height of an elimination forest of $G$.


The following observation is straightforward.

\begin{obs}
    For every annotated graph $(G,S)$ and for every positive integer $k$,
    we have $\td(G,S) \leq k$ if and only if 
    there exists a set $S' \subset V(G)$ with $S \subset S'$,
    and a rooted forest $F$ of height at most $k$ with $G[S']$ being a subgraph of the closure of $F$,
    such that,
    for every component $C$ of $G - S'$,
    there exists a root-to-leaf path in $F$ containing the neighborhood of $V(C)$ in $G$.
\end{obs}

Hodor, La, Micek, and Rambaud~\cite[Theorem~6]{quickly-excluding-apex-forest} proved that if for an annotated graph $(G,S)$, we have $\td(G,S) > \binom{k}{2}$, then $(G,S)$ contains $(P_k,V(P_k))$ as a minor.
Later, this quadratic bound was improved to linear by Claus, Hodor, Joret, and Morin~\cite[Theorem~3]{as-quickly-as-possible}.

\begin{theorem}[{\cite[Theorem~3]{as-quickly-as-possible}}]
    For every annotated graph $(G,S)$ and for every positive integer $k$, if $\td(G,S) \geq 2k-1$, then $G$ contains an $S$-rooted model of a path on $k$ vertices.
\end{theorem}

As already mentioned, a drawback of characterizing classes of bounded treedepth by forbidding paths as minors (or equivalently subgraphs) is that they have logarithmic treedepth themselves.
Kawarabayashi and Rossman~\cite{Kawarabayashi2021} found obstructions for bounded treedepth that have polynomial treedepth.
The precise bounds were later improved by Czerwiński, Nadara, and Pilipczuk, see~\Cref{th:originalABC} below.
Recall that a \defin{full binary tree} is a rooted tree where each vertex has either two or no children and all leaves have the same depth.

\begin{theorem}[{\cite[Corollary 1.5]{ImprovedCNP}}]\label{th:originalABC}
    There exists an absolute constant $C$ such that for all positive integers $a$, $b$, and $c$, if a graph $G$ has treedepth at least $C \cdot abc$, then one of the following conditions hold:
    \begin{enumerate}
        \item $G$ has treewidth at least $a$,
        \item $G$ contains a subdivision of a full binary tree of vertex-height $b$ as a subgraph, or
        \item $G$ contains a path on $2^c$ vertices as a subgraph.
    \end{enumerate}
\end{theorem}

We derive an annotated variant of~\Cref{th:originalABC}.

\begin{theorem}\label{th:focusedABC}
    There exists an absolute constant $C$ such that for all positive integers $a$, $b$, and $c$, if an annotated graph $(G,S)$ satisfies $\td(G,S) \geq C \cdot abc$, then one of the following conditions hold:
    \begin{enumerate}
        \item $\tw(G,S) \geq a$, \label{item:tw}
        \item $G$ contains a full binary tree of vertex-height $b$ as an $S$-rooted minor, or \label{item:tree}
        \item $G$ contains a path on $2^c$ vertices as an $S$-rooted minor. \label{item:path}
    \end{enumerate}
\end{theorem}


\begin{lemma} [based on {\cite[Theorem 4.1]{ImprovedCNP}}] \label{lem:tw-or-large-tree-td-minor}
    Let $(G,S)$ be an annotated graph, and let $a$ and $d$ be positive integers.
    If $\td(G,S) > ad$, then either $\tw(G,S) \geq a$ or $G$ contains an $S$-rooted model of a tree of treedepth greater than $d$. 
\end{lemma}
\begin{proof}
    By contraposition, we prove that if $\tw(G, S) < a$ and there is no $S$-rooted model of a tree $T$ in $G$ with $\td(T ) > d$, then $\td(G, S) \leq ad$.
    The proof is by induction on $d$.
    If $d = 1$, then the fact that there is no $S$-rooted model of a tree $T$ in $G$ with $\td(T) > 1$ implies that every component of $G$ has at most one vertex in $S$.
    This implies that $\td(G,S) \leq 1$, as desired.

    Thus, we assume that $d > 1$ and without loss of generality, we also assume that $G$ is connected.
    Let $\calB$ be the collection of all connected subgraphs of $G$ that contain an $S$-rooted model of a tree of treedepth exactly $d$.
    
    We claim that each two distinct $B_1$ and $B_2$ in $\calB$ intersect.
    Suppose to the contrary that some distinct $B_1$ and $B_2$ in $\calB$ are disjoint.
    For each $i \in [2]$, let $\calM_i$ be an $(S\cap V(B_i))$-rooted model of a tree $T_i$ with $\td(T_i) = d$.
    Since $G$ is connected, and $B_1$ and $B_2$ are disjoint, by extending one of the branch sets of $\calM_1$ and combining with $\calM_2$, we obtain an $S$-rooted model $\calM$ of a tree $T$, which is obtained from the disjoint union of $T_1$ and $T_2$ by adding an edge between an appropriate vertex of $T_1$ and an appropriate vertex of $T_2$.
    Observe that $T$ has treedepth at least $d+1$.
    This contradicts the assumption, and proves that each two distinct $B_1$ and $B_2$ in $\calB$ intersect.

    Let $\big(T,(W_x\mid x\in V(T))\big)$ be a tree decomposition of $(G,S)$ of width less than $a$.
    By~\Cref{lem:tw:erdos:posa} for $k=2$, there exists $t \in V(T)$ such that $W_t \cap V(B) \neq \emptyset$ for every $B \in \calB$.
    Therefore, by the definition of $\calB$, every tree $T'$ such that $G - W_t$ contains an $S$-rooted model of $T'$ satisfies $\td(T') < d$.
    By the inductive hypothesis, the treedepth of $(G - W_t, S\setminus W_t)$ is at most $a \cdot (d-1)$.
    Finally,
    \[\td(G,S) \leq \td(G - W_t,S\backslash W_t) + |W_t| \leq a\cdot (d-1) + a \leq ad. \qedhere\]
\end{proof}

The proof of \cref{th:focusedABC} makes also use of the following two results.

\begin{lemma} [{\cite[Lemma 1.6]{ImprovedCNP}}] \label{lem:large-subcubic-subtree}
    There exists an absolute constant $D > 0.43$ such that for every positive integer $d$, every tree of treedepth more than $d$ contains a subcubic subtree of treedepth more than $Dd$.
\end{lemma}

\begin{lemma}[{\cite[Lemma 5.10]{Kawarabayashi2021}}] \label{lem:tree-or-path}
    For all positive integers $b$ and $c$, every tree with $\td(T) > b(c+2)$ either contains a subdivision of a full binary tree of vertex-height $b$ as a subgraph or contains a path on $2^c$ vertices as a subgraph.
\end{lemma}

\begin{proof}[Proof of~\Cref{th:focusedABC}]
    Let $D$ be the constant from~\Cref{lem:large-subcubic-subtree}.
    We show that if an annotated graph $(G,S)$ has treedepth more than $\lceil\frac{1}{D}b(c+2)\rceil a$, then one of \ref{item:tw}, \ref{item:tree}, and \ref{item:path} holds.
    This implies the statement with $C = 3\cdot \frac{1}{D} + 1$.
    With  $d = \lceil\frac{1}{D}b(c+2)\rceil$, by \Cref{lem:tw-or-large-tree-td-minor}, $\td(G,S) > da$ implies that either $\tw(G,S) \geq a$ or $G$ contains an $S$-rooted model of a tree $T$ in $G$ with $\td(T) > d$. 
    If the former holds, then we obtain~\ref{item:tw}, hence assume the latter.
    By~\Cref{lem:large-subcubic-subtree}, $T$ contains a subcubic subtree $T'$ with $\td(T') > Dd \geq b(c+2)$.
    By~\Cref{lem:tree-or-path}, $T'$ contains either a subdivision of a full binary tree of vertex-height $b$ as a subgraph or a path on $2^c$ vertices as a subgraph.
    Since $G$ contains an $S$-rooted model of $T$, $G$ contains either an $S$-rooted model of a full binary tree of vertex-height $b$, or an $S$-rooted model of a path on $2^c$ vertices.
    In the former case, \ref{item:tree} holds, and in the latter case, \ref{item:path} holds, which ends the proof.
\end{proof}

\subsection{\texorpdfstring{$2$}{2}-treedepth}

We next give the definition of $2$-treedepth, which was introduced by Huynh, Joret, Micek, Seweryn, and Wollan~\cite{HJMSW22}.
The \defin{$2$-treedepth} of a graph $G$, denoted by \defin{$\td_2(G)$}, is defined recursively as
\begin{enumerate}[label={\normalfont (t\arabic*)}]
    \item $\td_2(G) = 0$ if $G$ is the null graph;
    \item $\td_2(G) = \max_{B} \td_2(B)$ over all the blocks\footnote{A \defin{block} of a graph $G$ is a maximal subgraph of $G$ which is either $2$-connected, a single edge, or an isolated vertex.} $B$ of $G$, if $G$ has more than one block; and
    \item $\td_2(G) = 1 + \min_{u \in V(G)} \td_2(G-u)$ if $G$ has exactly one block.
\end{enumerate}
The following simple observation follows from studying the definitions.

\begin{obs}\label{obs:td2}
    For every annotated graph $(G,S)$ with $S \neq \emptyset$ and $2$-connected $G$, we have
    \[\td_2(G,S) = 1 + \min_{u \in V(G)} \td_2(G -\{u\}, S \setminus \{u\}).\]
\end{obs}

For every positive integer $k$, let
\defin{$L_k$} be the $2 \times k$ grid,
and let \defin{$T_k$} be the vertex set of its first row.
See \cref{fig:Lk}.
It is easy to  see that $\td_2(L_{2^\ell}, T_{2^\ell}) \geq \ell$:
by induction, this is clear for $\ell=1$,
and for $\ell \geq 2$, by \cref{obs:td2},
there is a vertex $u$ such that $\td_2(L_{2^\ell}-u, T_{2^\ell} \setminus \{u\}) + 1 = \td_2(L_{2^\ell}, T_{2^\ell})$.
Since $(L_{2^{\ell-1}}, T_{2^{\ell-1}})$ is a minor of $(L_{2^\ell}-u, T_{2^\ell} \setminus \{u\})$,
we deduce $\td_2(L_{2^\ell}, T_{2^\ell}) \geq \td_2(L_{2^{\ell-1}}, T_{2^{\ell-1}}) + 1$.
This proves that the family $(L_k,T_k)$ for $k \in \mathbb{N}$ has unbounded $2$-treedepth.

\begin{figure}[tp]
        \centering
        \begin{tikzpicture}[every node/.style={draw, circle, fill, minimum size = 4pt, inner sep = 0pt}, scale=1]
            \foreach \x in {0,...,5}{
                 \node[rectangle, red] (u\x) at (\x,1) {};
                 \node (v\x) at (\x,0) {};
                 \draw (u\x) -- (v\x);
             }
             \foreach \x [evaluate=\x as \xx using {int(\x+1)}] in {0,...,4}{
                  \draw (u\x) -- (u\xx);
                  \draw (v\x) -- (v\xx);
             }
        \end{tikzpicture}
        \caption{The annotated graph $(L_6,T_6)$. 
            Vertices in $T_6$ are depicted in red.}
        \label{fig:Lk}
\end{figure}
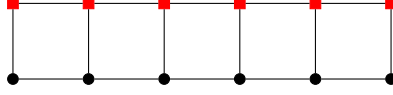



\begin{theorem}\label{thm:rooted:tdtwo:ladder}
    There exists a function $f$ such that for every annotated graph $(G,S)$ and for every positive integer $k$, if $\tdtwo(G,S) \geq f(k)$, then $(G,S)$ contains a model of $(L_k,T_k)$.
\end{theorem}

To prove \Cref{thm:rooted:tdtwo:ladder}, we introduce a few supplementary definitions (see~\Cref{fig:supp-def-td2} for their illustration).
Let $G$ be a graph, let $P$ be a path in $G$, and let $\calH$ be a family of pairwise disjoint subgraphs of $G$, each intersecting $P$.
For every $H \in \calH$ we define \defin{$J^{\calH}_P(H)$} to be the shortest subpath of $P$ containing $V(P) \cap V(H)$,
and we define \defin{$I^{\calH}_P(H)$} to be the set of maximal subpaths of $P$ such that each subpath has both endpoints in $V(H)$ 
and contains no vertices of any $H' \in \calH \setminus \{H\}$.
Note that $\bigcup \{V(Q) \mid Q \in I^\calH_P(H)\} \subset V(J^\calH_P(H))$.
When $\mathcal{H}$ is clear from the context, we will simply write
$J_P(H)$ for $J^\calH_P(H)$ and $\mathcal{I}_P(H)$ for $\mathcal{I}^\calH_P(H)$.

\begin{figure}[tp]
    \centering
    \begin{tikzpicture}[every node/.style={circle, minimum size = 4pt, inner sep = 0pt}]
        \draw[opacity=0.75, fill, Chartreuse3] (0.66,-0.33) rectangle (9.33,0.33);
        \node[Chartreuse3] at (1.5,-1) {$\mathcal{J}^\mathcal{H}_P(H_1)$};

        \draw[opacity=0.75, fill, MediumPurple1] (4.8,-0.2) rectangle (7.2, 0.2);
        \node[MediumPurple1] at (6,-1) {$\mathcal{J}^\mathcal{H}_P(H_2)$};

        \draw[opacity=0.75, fill, Chartreuse4] (0.8, -0.2) rectangle (4.2, 0.2);
        \draw[opacity=0.75, fill, Chartreuse4] (7.8, -0.2) rectangle (9.2, 0.2);
        \node[inner sep=1mm, rectangle, Chartreuse4] (labelI) at (8.5,-2) {$\mathcal{I}^\mathcal{H}_P(H_1)$};

        \draw[Chartreuse4, -{Latex[length=3mm]}, bend left, in=-60, out=-180, relative=false] (labelI) to (2.75, -0.5);
        \draw[Chartreuse4, -{Latex[length=3mm]}] (labelI) to (8.5, -0.5);
        
        \foreach \i in {0,...,10}{
            \node[circle, fill] (p\i) at (\i,0) {};
        }
        \node (p-1) at (-1,0) {};
        \node (p11) at (11,0) {};
        \foreach \i [evaluate=\i as \j using {int(\i+1)}] in {-1,...,10}{
            \draw (p\i) -- (p\j);
        }
        \node at (11+0.33,-0.33) {$P$};

        \node[black, fill] (h1) at (2.5,1) {};
        \node[black, fill] at (p1) {};
        \node[black, fill] at (p4) {};
        \draw[black] (h1) -- (p1);
        \draw[black] (h1) -- (p4);
        \node[black, fill] (h2) at (4.5,2) {};
        \draw[black] (h1) -- (h2);
        \node[black, fill] (h3) at (6.5,2) {};
        \draw[black] (h2) -- (h3);
        \node[black, fill] (h4) at (8.5,1) {};
        \draw[black] (h3) -- (h4);
        \node[black, fill] at (p8) {};
        \node[black, fill] at (p9) {};
        \draw[black] (h4) -- (p8);
        \draw[black] (h4) -- (p9);
        \node[black] at (8,2) {$H_1$};

        \node[black, fill] (h') at (6,1) {};
        \node[black, fill] at (p5) {};
        \node[black, fill] at (p6) {};
        \node[black, fill] at (p7) {};
        \draw[black] (h') -- (p5);
        \draw[black] (h') -- (p6);
        \draw[black] (h') -- (p7);
        \node[black] at (6.85,0.85) {$H_2$};
    \end{tikzpicture}
    \caption{Illustration for the definitions of $\mathcal{J}^\mathcal{H}_P(H)$ and $\mathcal{I}^\mathcal{H}_P(H)$.}
    \label{fig:supp-def-td2}
\end{figure}
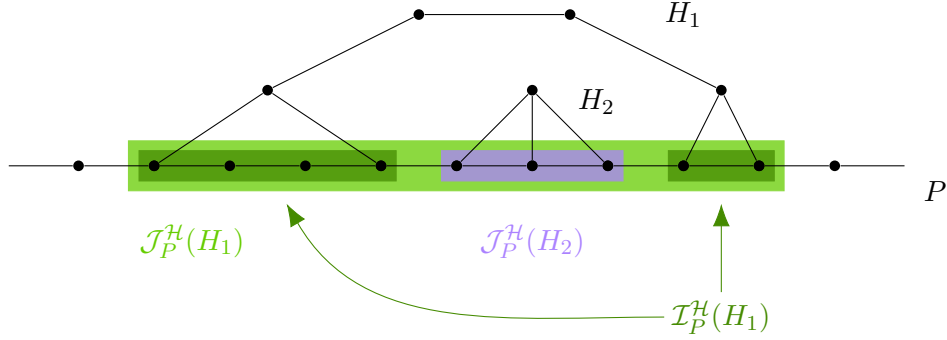



The goal of the following lemmas is to prove \Cref{lemma:td2:packing_of_steps_makes_a_ladder}, 
which will be used several times in the proof of \Cref{thm:rooted:tdtwo:ladder}.
This is inspired by~\cite[Lemma~16]{rambaud2025excludingrectangulargrid}.

\begin{lemma}\label{lem:S_B}
    Let \((G,S)\) be an annotated graph. 
    For every block \(B\) of \(G\), let 
    \[
    S_B:=\{v\in V(B)\mid \text{ there exists a \(v\)--\(S\) path \(P\) in \(G\) with }
    V(P)\cap V(B)=\{v\}\}.
    \]
    Then
\[
\td_2(G,S)= \max\{\td_2(B,S_B)\mid  B\text{ is a block of }G\}.
\]
Moreover, if an annotated graph \((H,T)\) is a minor of \((B,S_B)\) for some block $B$ of $G$, then \((H,T)\) is a minor of \((G,S)\).
\end{lemma}
\begin{proof}
We first prove an auxiliary observation. Suppose that $G_1$ and $G_2$ are
connected subgraphs of $G$ such that $G = G_1 \cup G_2$ and $|V(G_1) \cap V(G_2)| = 1$.
Let $v$ be the common vertex of $G_1$ and $G_2$.
For $\{i,j\} = [2]$, define
\[
S_i=(S\cap V(G_i))\cup
\begin{cases}
\{v\},&\text{if }S\cap (V(G_j)\setminus\{v\})\neq\emptyset,\\
\emptyset,&\text{otherwise}.
\end{cases}
\]
We claim that
\begin{equation}\label{eq:td2-one-sum}
\td_2(G,S)=\max\{\td_2(G_1,S_1),\td_2(G_2,S_2)\}.
\end{equation}

If $S=\emptyset$, then both sides of~\eqref{eq:td2-one-sum} are equal
to $0$, so suppose that $S\neq\emptyset$.
We first prove the upper bound. 
Suppose first that $S\cap V(G_1)=\emptyset$. Then $v\notin S$, $S_1=\{v\}$, and
$S_2=S$. Let $S_2\subset S_2'$ be such that $\td_2(\torso_{G_2}(S'_2))=\td_2(G_2,S_2)$.
Since $S'_2\subseteq V(G_2)$, we have $\torso_G(S'_2)=\torso_{G_2}(S'_2)$.
Indeed, a path between two vertices of $G_2$ cannot use
$G_1-\{v\}$ without visiting $v$ twice. Hence $\td_2(G,S)\leq \td_2(G_2,S_2)$.
Moreover, $\td_2(G_1,S_1)=1$, while
$\td_2(G_2,S_2)\geq 1$, and consequently $\td_2(G,S)
  \leq
  \max\{\td_2(G_1,S_1),\td_2(G_2,S_2)\}$.
The case $S\cap V(G_2)=\emptyset$ is symmetric.

We may therefore assume that $S$ intersects both $V(G_1)$ and
$V(G_2)$. 
For each $i\in[2]$, let $S_i \subset S_i'$ be such that $\td_2(\torso_{G_i}(S_i'))=\td_2(G_i,S_i)$.
Let $S' = S_1' \cup S_2'$ and note that $v\in S_1'\cap S_2'$.
Additionally,
\[
\torso_G(S')=\torso_{G_1}(S_1')\cup\torso_{G_2}(S_2'),
\]
and the two graphs on the right-hand side intersect only in $v$.
Since every block of the graph on the right-hand side is contained in one of its two parts, the definition of
$2$-treedepth gives
\[
\td_2(\torso_G(S'))
=
\max_{i\in[2]}\td_2(\torso_{G_i}(S_i')).
\]
Consequently,
\[
\td_2(G,S)
\leq
\max\{\td_2(G_1,S_1),\td_2(G_2,S_2)\}.
\]
The reverse inequality follows from \cref{obs:monotonicity}, thus, we obtain \eqref{eq:td2-one-sum}.

We now prove the first assertion of the lemma. 
Without loss of generality, we may assume that $G$ is connected.

We proceed by induction on the number of blocks of $G$. If $G$ has a
single block $B$, then $B=G$ and $S_B=S$, so there is nothing to prove.
Suppose now that $G$ has at least two blocks. Let $B_0$ be a leaf block
of the block-cut tree of $G$, and let $v$ be the unique cutvertex of
$G$ contained in $B_0$. Let $G'$ be the union of all blocks of $G$
different from $B_0$. Then $G=B_0\cup G'$ and $V(B_0)\cap V(G')=\{v\}$.
Let $S_1$ and $S_2$ be the sets defined above \eqref{eq:td2-one-sum} for $G_1 = B_0$ and $G_2 = G'$. 
By the definition of $S_{B_0}$, we have $S_1=S_{B_0}$.
Moreover, for every block $B$ of $G'$, the set obtained by projecting
$S_2$ onto $B$ is precisely $S_B$. Indeed, a path witnessing this
property either ends at a vertex of $S\cap V(G')$, in which case it
also witnesses the corresponding property in $G$, or it ends at $v$
because $S$ has a vertex in $V(B_0)\setminus\{v\}$; in the latter case
it can be extended through $B_0$ to a vertex of $S$. Conversely, a
path in $G$ witnessing that a vertex of $B$ belongs to $S_B$ can be
truncated at $v$ if its endpoint in $S$ lies in
$V(B_0)\setminus\{v\}$.

Therefore, by \eqref{eq:td2-one-sum} and the induction hypothesis,
\begin{align*}
\td_2(G,S)
&=
\max\{\td_2(B_0,S_{B_0}),\td_2(G',S_2)\}\\
&=
\max\{\td_2(B,S_B):B\text{ is a block of }G\}.
\end{align*}
This proves the first assertion.
The \q{moreover} part is straightforward.
\end{proof}

\begin{lemma}\label{lem:adhesion}
    Let $k$ be an integer with $k \geq 2$,
    let $G$ be a graph, 
    let $S \subseteq V(G)$, 
    and let $\calW = \big(T, (W_t \mid t \in V(T))\big)$ be a tree decomposition of $(G, S)$ with width less than $k$.
    Let $\calH$ be a family of pairwise disjoint connected subgraphs of $G$ each intersecting $S$. 
    There exists a family $\calF \subset \calH$ such that 
    \begin{enumerate}
        \item $|\calF| \geqslant \sqrt{|\calH|/k}$, and 
        \item $|\mathcal{I}^{\calF}_P(F)| \leq k+1$ for every $F \in \calF$ and path $P$ in $G$.
    \end{enumerate}
\end{lemma}

\begin{proof}
    Let $\ell = \lceil \sqrt{|\mathcal{H}|/k} \rceil$.
    If $\ell \leq 1$, then taking for $\mathcal{F}$ any $\ell$-subset of $\mathcal{H}$ witnesses the statement.
    Now suppose that $\ell \geq 2$.

    Note that $\mathcal{W}$ has adhesion at most $k$.
    Since the members of $\mathcal{H}$ are pairwise disjoint,
    there is no set of less than $|\mathcal{H}|$ vertices intersecting every member of $\mathcal{H}$.
    Therefore, there is no set $X \subseteq V(T)$ with $|X| \leq \frac{|\mathcal{H}|-1}{k}$ such that $\bigcup_{x \in X} W_x$ intersects every member of $\mathcal{H}$.
    By \Cref{lem:tw:erdos:posa}, there is $\mathcal{H}' \subseteq \mathcal{H}$
    such that
    \begin{enumerate}
        \item $|\mathcal{H}'| \geq \frac{|\mathcal{H}|}{k}$, and
        \item the sets $\mathcal{W}[H]$ for $H \in \mathcal{H}'$ are pairwise disjoint.
    \end{enumerate}
    Consider a minimal subtree $F$ of $T$ that intersects every $\calW[H]$ for $H \in \mathcal{H}'$.

    We now distinguish several cases.
    First suppose that $F$ has at least $\ell$ leaves.
    Then, let $\mathcal{F}$ be the set of all the members $H$ of $\mathcal{H}'$ such that $\calW[H]$ contains a leaf of $F$.
    Consider some $H \in \mathcal{F}$.
    There is an adhesion in $\mathcal{W}$
    separating $V(H)$ from $\bigcup_{H' \in \mathcal{F}, H' \neq H} V(H')$ in $G$.
    Let $Z$ be this adhesion.
    Since for every two consecutive $I,I' \in \mathcal{I}^\mathcal{F}_P(H)$ along $P$,
    the subpath of $P$ between $I$ and $I'$ intersects $\bigcup_{H' \in \mathcal{F}, H' \neq H} V(H')$,
    this subpath crosses $Z$ at least twice,
    we deduce that
    \[
        |\mathcal{I}^\mathcal{F}_P(H)| \leq |Z|/2+1 \leq k+1.
    \]

    Now suppose that there is a path in $F$ and $\ell$ members $H_1,\dots,H_\ell$ of $\calH'$ such that this path intersects $\calW[H_i]$ for every $i \in [\ell]$.
    We take $\mathcal{F} = \{H_1, \dots, H_\ell\}$.
    Consider $i \in [\ell]$.
    Then, there are at most two adhesions in $\mathcal{W}$
    separating $V(H_i)$ and $\bigcup_{j \in [\ell], j \neq i} V(H_j)$ in $G$.
    Let $Z$ be the union of these two adhesions.
    Since for every two consecutive $I,I' \in \mathcal{I}^\mathcal{F}_P(H_i)$ for some $i \in [\ell]$ along $P$,
    the subpath of $P$ between $I$ and $I'$ intersects $\bigcup_{j \in [\ell], j \neq i} V(H_j)$,
    this subpath crosses $Z$ twice,
    we deduce that
    \[
        |\mathcal{I}^\mathcal{F}_P(H_i)| \leq |Z|/2+1 \leq k+1.
    \]

    In the last case,
    every path in $F$ intersects at most $\ell-1$ sets $\calW[H]$ for $H \in \mathcal{H}'$,
    and there are at most $\ell-1$ leaves in $F$.
    Let $r$ be any vertex of $F$.
    Then, for every leaf $x$, every $r$-$x$ path in $F$
    intersects at most $\ell-1$ sets $\calW[H]$ for $H \in \mathcal{H}'$.
    Since $F$ is covered by the paths between $r$ and the leaves,
    we deduce that
    \[
        |\mathcal{H}'| \leq (\ell-1)^2 < \tfrac{|\mathcal{H}|}{k},
    \]
    a contradiction.
\end{proof}


We will also need the celebrated Erd\H{o}s-Szekeres Theorem:
\begin{theorem}[Erd\H{o}s-Szekeres Theorem]\label{thm:erdos:szekeres}
    For every positive integer $k$ and every permutation
    $\pi$ of $[k^2+1]$, there exists $I \subseteq [k^2+1]$ of size $k+1$
    such that $\pi\vert_I$ is either decreasing or increasing.
\end{theorem}

\begin{lemma}\label{lem:tdtwo:ladder:technical:newversion}
    There is a function $f_{\ref{lem:tdtwo:ladder:technical:newversion}}\colon \posint^3 \to \posint$ such that
    the following holds.
    Let $k$, $\ell_1$, and $\ell_2$ be positive integers,
    let $G$ be a graph,
    let $P_1$ and $P_2$ be disjoint paths in $G$,
    and let $S \subset V(G)$.
    Let $\calH$ be a family of connected, pairwise disjoint subgraphs of $G$ such that for every $H \in \mathcal{H}$,
    \begin{enumerate}
        \item $V(H) \cap S \neq \emptyset$,
        \item $1 \leq |\mathcal{I}_{P_1}(H)| \leq \ell_1$, and 
        \item $1 \leq |\mathcal{I}_{P_2}(H)| \leq \ell_2$.
    \end{enumerate}
    If $|\calH| \geqslant f_{\ref{lem:tdtwo:ladder:technical:newversion}}(k, \ell_1, \ell_2)$, 
    then $(L_k, T_k)$ is a minor of $(G,S)$. 
\end{lemma}

\begin{proof}
We set 
\[
f_{\ref{lem:tdtwo:ladder:technical:newversion}}(k,\ell_1,\ell_2) =
\begin{cases}
    2((k-1)^2+1) & \textrm{if $\ell_1=\ell_2=1$,} \\
    f_{\ref{lem:tdtwo:ladder:technical:newversion}}(k,1,\ell_2)(f_{\ref{lem:tdtwo:ladder:technical:newversion}}(k,\ell_1-1,\ell_2)+1)-1& \textrm{if $\ell_1\geq 2$,} \\
    f_{\ref{lem:tdtwo:ladder:technical:newversion}}(k,\ell_1,1)(f_{\ref{lem:tdtwo:ladder:technical:newversion}}(k,\ell_1,\ell_2-1)+1)-1 & \textrm{if $\ell_1=1$ and $\ell_2\geq 2$.}
\end{cases}
\]

We call a tuple $(k,\ell_1,\ell_2,P_1,P_2,\mathcal{H})$ satisfying the assumptions of the lemma an \emph{instance}.
We prove the statement of the lemma by induction on $\ell_1 + \ell_2 + |V(G)|$.
Let $(k,\ell_1,\ell_2,P_1,P_2,\mathcal{H})$ be an instance with $|\mathcal{H}| \geq f_{\ref{lem:tdtwo:ladder:technical:newversion}}(k,\ell_1,\ell_2)$.

If there exists
$H \in \mathcal{H}$ and $I \in \mathcal{I}_{P_1}(H) \cup \mathcal{I}_{P_2}(H)$ with $|V(I)|\geq 2$,
then we can call the induction hypothesis on the instance obtained by contracting an edge in the path $I$: if there is a model of $(L_k,T_k)$ in the contracted graph, there is also in the original one.
Hence we can assume the following.
\begin{equation}\label{eq:ladder_technical:intersections_of_size_1}
    \text{For every $H \in \mathcal{H}$, for every $I \in \mathcal{I}_{P_1}(H) \cup \mathcal{I}_{P_2}(H)$, $|V(I)|=1$.} 
\end{equation}
Similarly, if there is $H \in \mathcal{H}$ and an edge $e=uv$ in $H$ with $v \not\in V(P_1) \cup V(P_2)$,
then we can contract $e$ and call the induction hypothesis on this smaller instance.
Hence, using the fact that every member of $\mathcal{H}$ is connected,
we can assume the following.
\begin{equation}\label{eq:ladder_technical:H_in_P1cupP2}
    \text{For every $H \in \mathcal{H}$, $V(H) \subseteq V(P_1) \cup V(P_2)$.} 
\end{equation}

First suppose $\ell_1=\ell_2=1$.
Then, for every $H \in \mathcal{H}$ and $i \in [2]$,
$V(H) \cap V(P_i)$ consists of one vertex because of \eqref{eq:ladder_technical:intersections_of_size_1}.
By \eqref{eq:ladder_technical:H_in_P1cupP2}, this implies that
every $H \in \mathcal{H}$ consists of a single edge, that we denote by $p_1(H)p_2(H)$ with $p_1(H) \in V(P_1)$ and $p_2(H) \in V(P_2)$.
Let $\mathcal{H}_i = \{H \in \mathcal{H} \mid p_i(H) \in S\}$ for each $i \in [2]$.
Since $\mathcal{H} = \mathcal{H}_1 \cup \mathcal{H}_2$,
there exists $i\in[2]$ such that $|\mathcal{H}_i| \geq |\mathcal{H}|/2 \geq (k-1)^2+1$.
Without loss of generality assume $i=1$.
Then, by the Erd\H{o}s-Szekeres Theorem (\Cref{thm:erdos:szekeres}),
there is a subfamily $\mathcal{H}' \subseteq \mathcal{H}_1$ of size $k$ 
such that the vertices $p_1(H)$ and $p_2(H)$ for $H \in \mathcal{H}'$
appear in the same or reverse order along $P_1,P_2$.
By possibly reversing the order of $P_2$, we can assume they appear in the same order.
Let $H_1, \dots, H_k$ be the members of $\mathcal{H}'$ such that
$p_i(H_1), \dots, p_i(H_k)$ appear in this order along $P_i$, for each $i \in [2]$.
Then, for every $i \in [2]$, for every $j \in [k-1]$,
let $B_{i,j}$ denote the vertex set of the subpath of $P_i$ starting at $p_i(H_j)$ and ending at the predecessor of $p_i(H_{j+1})$,
and let $B_{i,k} = \{p_i(H_k)\}$.
Then, $(B_{i,j} \mid (i,j) \in [2] \times [k])$ is a model of $(L_k,T_k)$ in $(G,S)$.

Now suppose that $\ell_1 \geq 2$ or $\ell_2 \geq 2$ and without loss of generality, assume $\ell_1 \geq 2$.

Consider the family of the subpaths $J_{P_1}(H)$ for $H \in \mathcal{H}$.
If there exists $\mathcal{H}' \subseteq \mathcal{H}$ of size at least $f_{\ref{lem:tdtwo:ladder:technical:newversion}}(k,1,\ell_2)$ for which the paths $J_{P_1}(H)$ for $H \in \mathcal{H}'$
are pairwise disjoint, then 
$|I^{\mathcal{H}'}_{P_1}(H)| = 1$ for each $H \in \mathcal{H}'$,
and so we can call the induction hypothesis on the instance $(k,1,\ell_2, P_1, P_2, \mathcal{H}')$.
Otherwise, by the Helly property of subpaths of $P_1$,
there exists $X \subseteq V(P_1)$ of size at most $f_{\ref{lem:tdtwo:ladder:technical:newversion}}(k,1,\ell_2)-1$
such that $X$ intersects $V(J_{P_1}^\calH(H))$ for every $H \in \mathcal{H}$.
Let $Q_1, \dots, Q_m$ be the components of $P_1-X$.
Recall that every $H \in \mathcal{H}$ intersects $P_1$,
and so at least one of the sets $V(Q_1), \dots, V(Q_m)$ and $\{x\}$ for $x \in X$.
Since at most $|X|$ members of $\mathcal{H}$ intersect $X$,
we deduce by the pigeonhole principle
that there exists $j \in [m]$
and $\mathcal{H}' \subseteq \mathcal{H}$ such that
\begin{enumerate}
    \item $|\mathcal{H}'| \geq \frac{|\mathcal{H}|-|X|}{m} \geq f_{\ref{lem:tdtwo:ladder:technical:newversion}}(k,\ell_1-1,\ell_2)$,
    \item for every $H \in \mathcal{H}'$, $H$ is disjoint from $X$, and
    \item for every $H \in \mathcal{H}'$, $H$ intersects $Q_j$.
\end{enumerate}

Let $H \in \mathcal{H}'$.
By \eqref{eq:ladder_technical:intersections_of_size_1},
every $I \in \mathcal{I}_{P_1}(H)$ consists of a single vertex.
Hence, there is $I \in \mathcal{I}_{P_1}(H)$ that is disjoint from $Q_j$.
This implies that
\[
    |\mathcal{I}_{Q_j}^{\mathcal{H}'}(H)| \leq |\mathcal{I}_{P_1}^{\mathcal{H}}(H)|-1 \leq \ell_1-1.
\]
Therefore, we can call the induction hypothesis on
the instance $(k,\ell_1-1,\ell_2,Q_j,P_2,\mathcal{H}')$.
This proves the lemma.
\end{proof}

\begin{lemma}\label{lemma:td2:packing_of_steps_makes_a_ladder}
    There is a function $f_{\ref{lemma:td2:packing_of_steps_makes_a_ladder}} \colon \posint \to \posint$ such that
    the following holds.
    Let $k$ be a positive integer,
    let $G$ be a graph,
    let $S \subset V(G)$,
    let $P_1$ and $P_2$ be two disjoint paths in $G$,
    and let $\mathcal{H}$ be a family of pairwise disjoint connected subgraphs of $G$ such that,
    for every $H \in \mathcal{H}$,
    \begin{enumerate}
        \item $H$ intersects $S$,
        \item $H$ intersects $V(P_1)$, and
        \item $H$ intersects $V(P_2)$.
    \end{enumerate}
    If $|\mathcal{H}| \geq f_{\ref{lemma:td2:packing_of_steps_makes_a_ladder}}(k)$, then
    $(L_k,T_k)$ is a minor of $(G,S)$.
\end{lemma}

\begin{proof}
    Let $f_{\ref{lemma:td2:packing_of_steps_makes_a_ladder}}(k) = \big(f_{\ref{thm:rooted:gridminor}}(2k)\big)^3\cdot \big(f_{\ref{lem:tdtwo:ladder:technical:newversion}}(k,f_{\ref{thm:rooted:gridminor}}(2k)+1,f_{\ref{thm:rooted:gridminor}}(2k)+1)\big)^4$.
    Suppose for contradiction that $|\mathcal{H}| \geq f_{\ref{lemma:td2:packing_of_steps_makes_a_ladder}}(k)$
    and that there is no model of $(L_k,T_k)$ in $(G,S)$.
    By \Cref{thm:rooted:gridminor}, there is a tree decomposition of $(G,S)$
    of width less than $f_{\ref{thm:rooted:gridminor}}(2k)$.
    By applying \Cref{lem:adhesion} twice,
    we deduce that
    there exists $\mathcal{H}' \subseteq \mathcal{H}$ of size at least
    $f_{\ref{lem:tdtwo:ladder:technical:newversion}}(k,f_{\ref{thm:rooted:gridminor}}(2k)+1,f_{\ref{thm:rooted:gridminor}}(2k)+1)$
    such that
    $|\mathcal{I}_{P_i}^{\mathcal{H}'}(H)| \leq f_{\ref{thm:rooted:gridminor}}(2k)+1$ for all $H \in \mathcal{H}'$ and $i \in [2]$.
    By \Cref{lem:tdtwo:ladder:technical:newversion},
    we conclude that $(L_k,T_k)$ is a minor of $(G,S)$.
\end{proof}

\begin{lemma}\label{lemma:td2:main_induction}
    There is a function $f_{\ref{lemma:td2:main_induction}}\colon \posint^2 \to \posint$
    such that the following holds.
    For all positive integers $k$ and $\ell$,
    for every $2$-connected graph $G$,
    for every $S \subseteq V(G)$,
    for every $u_1,u_2 \in V(G)$ distinct,    
    if $\td_2(G,S) > f_{\ref{lemma:td2:main_induction}}(\ell,k)$,
    then either
    \begin{enumerate}[label={(\alph*)}]
        \item $(L_\ell,T_\ell)$ is a minor of $(G,S)$, or
        \item there is a model $(B_x \mid x \in V(L_k))$ of $L_k$ in $G$            such that 
            \begin{enumerate}
                \item for every column $C$ of $L_k$,
                    $S \cap (\bigcup_{c \in C}B_c) \neq \emptyset$, and
                \item the vertices of the first column of $L_k$ can be named $\{c_1,c_2\}$ so that $u_1 \in B_{c_1}$ and $u_2 \in B_{c_2}$.
            \end{enumerate}
    \end{enumerate}
\end{lemma}

\begin{proof}
    Let $f_{\ref{lemma:td2:main_induction}}(\cdot,\cdot)$ be defined by 
    induction as follows.
    For all positive integers $k,\ell$, let
    \[
        f_{\ref{lemma:td2:main_induction}}(\ell,k) =
        \begin{cases}
            0 & \textrm{if $k=1$,} \\
            2(f_{\ref{lemma:td2:packing_of_steps_makes_a_ladder}}(\ell)-1)
                \cdot f_{\ref{thm:rooted:gridminor}}(2\ell) 
            + f_{\ref{lemma:td2:main_induction}}(\ell,k-1) + 6 & 
                        \textrm{if $k \geq 2$.}
        \end{cases}
    \]

    We proceed by induction on $k$.
    Let $k$ and $\ell$ be positive integers,
    let $G$ be a $2$-connected graph,
    let $S \subseteq V(G)$,
    and let $u_1,u_2 \in V(G)$ distinct.
    Assume that $\td_2(G,S) > f_{\ref{lemma:td2:main_induction}}(\ell,k)$
    and that $(L_\ell,T_\ell)$ is not a minor of $(G,S)$.
    By \Cref{thm:rooted:gridminor},
    this implies that $\tw(G,S) < f_{\ref{thm:rooted:gridminor}}(2\ell)$.

    First suppose $k=1$.
    Since $G$ is $2$-connected, and because $S$ is not empty (since $\td_2(G,S)>0$),
    there exists $s \in S$ and two paths $P_1,P_2$ from $s$ to $\{u_1,u_2\}$
    in $G$ such that $V(P_1) \cap V(P_2) = \{s\}$ and $u_i\in V(P_i)$ for $i\in[2]$.
    Without loss of generality, $s \neq u_1$.
    Then, $(V(P_1)\setminus \{s\}, V(P_2))$ is a model of $(L_1,T_1)$ in $(G,S)$ as wanted.
    Thus, we assume that $k \geq 2$.

    Since $G$ is 2-connected, there exist two internally disjoint $u_1$--$u_2$ paths $P_1,P_2$.
    Let $\mathcal{H}$ be the family of all the connected subgraphs $H$ of $G$
    that intersect $S$ and at least one internal vertex of each path $P_i$, i.e., $V(H) \cap S \neq \emptyset$, $V(H) \cap (V(P_1) \setminus \{u_1,u_2\}) \neq \emptyset$, and $V(H) \cap (V(P_2) \setminus \{u_1,u_2\}) \neq \emptyset$.
    By \Cref{lemma:td2:packing_of_steps_makes_a_ladder},
    there are no $f_{\ref{lemma:td2:packing_of_steps_makes_a_ladder}}(\ell)$
    pairwise disjoint members of $\mathcal{H}$.
    Hence, by \Cref{lem:tw:erdos:posa},
    there exists $Z \subseteq V(G)$
    such that
    \begin{enumerate}
        \item $|Z| \leq (f_{\ref{lemma:td2:packing_of_steps_makes_a_ladder}}(\ell)-1) \cdot f_{\ref{thm:rooted:gridminor}}(2\ell)$, and
        \item $V(H) \cap Z \neq \emptyset$ for every $H \in \mathcal{H}$.
    \end{enumerate}

    Fix a vertex $s\in S$.
    Since $G$ is $2$-connected,
    there are two internally disjoint \stpaths{s}{\{u_1,u_2\}} $Q_1$ and $Q_2$ in $G$.
    Then, let $\mathcal{H}'$ be the family of all the connected subgraphs $H$ of $G$ that intersect $S$ and at least one vertex of each path $Q$ different than $s$, i.e.,
    $V(H) \cap S \neq \emptyset$, $V(H) \cap (V(Q_1) \setminus \{s\}) \neq \emptyset$, and $V(H) \cap (V(Q_2) \setminus \{s\}) \neq \emptyset$.
    Similarly as for $P_1,P_2$,
    there exists $Z' \subseteq V(G)$ such that
    \begin{enumerate}
        \item $|Z'| \leq (f_{\ref{lemma:td2:packing_of_steps_makes_a_ladder}}(\ell)-1) \cdot f_{\ref{thm:rooted:gridminor}}(2\ell)$, and
        \item $V(H) \cap Z' \neq \emptyset$ for every $H \in \mathcal{H}'$.
    \end{enumerate}
    We set $Z^\circ = Z\cup Z'\cup\{s,u_1,u_2\}$; keep in mind that $|Z^\circ|\le  2(f_{\ref{lemma:td2:packing_of_steps_makes_a_ladder}}(\ell)-1) \cdot f_{\ref{thm:rooted:gridminor}}(2\ell)+3$.
    We have
    \[
        \td_2(G-Z^\circ, S \setminus Z^\circ) 
        \geq \td_2(G,S) - |Z^\circ| 
        > f_{\ref{lemma:td2:main_induction}}(\ell,k-1)+2.
    \]
    Therefore, by \cref{lem:S_B}, there exists a block $G'$ of $G-Z^\circ$
    such that,
    for 
    \[
        S' = \{u \in V(G') \mid \text{there is a \stpath{u}{S} $P$ in $G-Z^\circ$ with $V(G') \cap V(P) = \{u\}$}\},
    \]
    we have $\td_2(G',S') > f_{\ref{lemma:td2:main_induction}}(\ell,k-1)+2$.
    Intuitively, $S'$ is the set of all vertices in $V(G')$ that are either in $S \setminus Z^\circ$ or can be connected to a vertex in $S \setminus Z^\circ$ via a sequence of blocks of $G-Z^\circ$.
    
    We next focus on fixing two vertices in $G'$ and applying the induction hypothesis in order to find a ladder $L_{k-1}$ that can be extended by one step in order to include $u_1$ and $u_2$ (and $s$).
        This is done as follows.
        Since $G$ is 2-connected, there exist two vertex-disjoint $\{u_1,u_2\}$--$V(G')$ paths $R_1,R_2$ in $G$.
        Among all such pairs, allowing their endpoints in $G'$ to vary, choose one minimizing
        \[
            |V(R_1)\cap Z^\circ|+|V(R_2)\cap Z^\circ|.
        \]
        Label the paths so that $u_i$ is an endpoint of $R_i$ for each
        $i\in[2]$, and let $u_i'$ be its endpoint in $V(G')$.
        Thus $V(R_i)\cap V(G')=\{u_i'\}$.
        
        Let $G_0$ be the subgraph of $G$ induced by the union of $V(G')$
        and the vertex sets of all components of
        $G-(Z^\circ\cup V(G'))$ having a neighbor in
        $V(G')\setminus\{u_1',u_2'\}$. See \cref{fig:td2_G0}.
        We claim that $(V(R_i)\setminus\{u_i'\})\cap V(G_0)=\emptyset$ for each $i \in [2]$.
        
        Suppose otherwise. Then some component $C$ used in the definition
        of $G_0$ meets $V(R_1)\cup V(R_2)$.
        Since $G'$ is a block of $G-Z^\circ$, the component $C$ has exactly
        one neighbor $v$ in $G'$, and
        $v\notin\{u_1',u_2'\}$.
        Choose a path $Q$ from $v$ to a vertex
        $w\in V(C)\cap(V(R_1)\cup V(R_2))$ such that
        $V(Q)\setminus\{v\}\subseteq V(C)$ and
        $Q$ meets $R_1\cup R_2$ only at $w$.
        Say $w\in V(R_i)$.
        Replace the $w$--$u_i'$ subpath of $R_i$ by $Q$, traversed from
        $w$ to $v$. This gives another pair of vertex-disjoint
        $\{u_1,u_2\}$--$V(G')$ paths.
        The discarded subpath must contain a vertex of $Z^\circ$:
        otherwise, since $C$ has no neighbor in $G'$ other than $v$,
        it could not reach $u_i'\ne v$ without meeting $G'$ before
        its endpoint. On the other hand, $Q$ is disjoint from $Z^\circ$.
        Consequently, the new pair has a strictly smaller value of
        $|V(R_1)\cap Z^\circ|+|V(R_2)\cap Z^\circ|$, a contradiction.
        This proves the claim.
        Additionally, $\{s,u_1,u_2\}\cap V(G_0)=\emptyset$,
        since $V(G_0)\cap Z^\circ=\emptyset$.


    \begin{figure}[tp]
        \centering
        \includegraphics{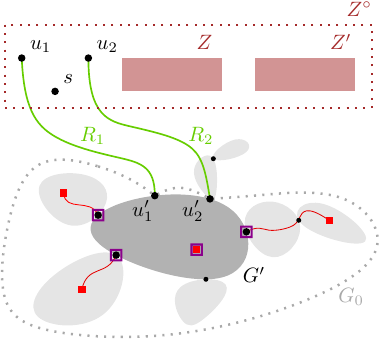}
        \caption{The construction of $G_0$ in the proof of \Cref{lemma:td2:main_induction}. The red vertices are the elements of $S$, and the purple ones are the elements of $S'$.}
        \label{fig:td2_G0}
    \end{figure}

    Since $\td_2(G',S'\setminus \{u'_1,u'_2\}) > f_{\ref{lemma:td2:main_induction}}(\ell,k-1)$,
    and by the induction hypothesis,
    there is a model $(B^0_x \mid x \in V(L_{k-1}))$ of $L_{k-1}$ in $G'$            such that 
    \begin{enumerate}
        \item for every column $C$ of $L_{k-1}$,
            $(S' \setminus \{u'_1,u'_2\}) \cap (\bigcup_{c \in C} B^0_c) \neq \emptyset$, and
        \item if $\{c^0_1,c^0_2\}$ is the first column of $L_{k-1}$,
            then $u'_1 \in B^0_{c^0_1}$ and $u'_2 \in B^0_{c^0_2}$, or vice-versa.
            Up to relabeling $c^0_1,c^0_2$, we assume $u'_1 \in B^0_{c^0_1}$
            and so $u'_2 \in B^0_{c^0_2}$.
    \end{enumerate} 
    By construction, this model can be extended into a model $(B^1_x \mid x \in V(L_{k-1}))$ of $L_{k-1}$ in $G_0$ such that
    \begin{enumerate}
        \item for every column $C$ of $L_{k-1}$,
            $S \cap (\bigcup_{c \in C} B^1_c) \neq \emptyset$, and
        \item $u_1' \in B^1_{c^0_1}$ and $u_2' \in B^1_{c^0_2}$.
    \end{enumerate}   

    By the properties of $Z$ and $Z'$,
    and because $V(G_0)$ is disjoint from $Z \cup Z'$,
    $G_0$ is not a member of $\mathcal{H} \cup \mathcal{H}'$.
    Since $V(G_0)$ intersects $S$,
    we deduce that
    \begin{enumerate}
        \item there is $i_0 \in [2]$ such that $P_{i_0}$ is disjoint from $V(G_0)$ and
        \item there is $j_0 \in [2]$ such that $Q_{j_0}$ is disjoint from $V(G_0)$.
    \end{enumerate}
    We fix such an $i_0 \in [2]$.
    We follow the path $R_1$ (resp. $R_2$) starting from $u_1'$ (resp, $u_2'$) until it reaches a vertex in $V(P_{i_0})$, which we denote by $w_1$ (resp. $w_2$). Such vertices $w_1$ and $w_2$ exist because $u_1$ (resp. $u_2$) are common vertices of $R_1$ (resp. $R_2$) and $P_{i_0}$. 
    Note that there is an $i\in[2]$ such that the subpath of $P_{i_0}$ starting from $u_i$ and finishing at $w_1$ does not contain $w_2$. We fix such $i\in[2]$.
    We consider the following path from $u_1'$ to $u_i$: We start following $R_1$ starting from $u_1'$, and then, when having reached $w_1$, we continue along $P_{i_0}$ until we reach $u_i$. We denote this path by $Y_1$. Also, we consider the graph $Y_2$ obtained by the union of $P_{i_0} - V(Y_1)$ and the subpath of $R_2$ starting from $u_2'$ and finishing at $w_2$. 
    See \Cref{fig:ladder_model}.
    Observe that both $Y_2- \{u_2'\}$ and $Y_1-\{u_1'\}$ are connected subgraphs of $G$ that are disjoint from $V(G_0)$, and have (at least) an edge between them. Our aim now is to extend the model of $L_{k-1}$ in $G_0$ to a model of $L_k$ in $G$ by including the sets $V(Y_2)- \{u_2'\}$ and $V(Y_1)-\{u_1'\}$. Still, we need an extra argument so as to ensure condition (i) from the statement of the lemma, i.e., extend one of the new bags of the model so as to include $s$.

    \begin{figure}[tp]
        \centering
        \includegraphics{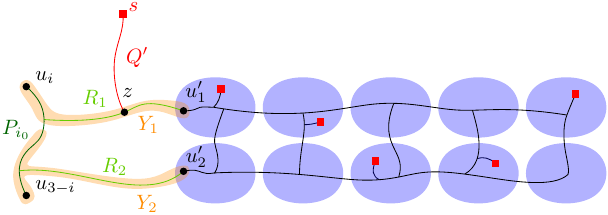}
        \caption{The construction of the desired model of $L_k$ in
            the proof of \Cref{lemma:td2:main_induction}.
            The vertices in $S$ are depicted by red squares.
            We are here in the case $j=1$, i.e. $z \in V(Y_1)$.}
        \label{fig:ladder_model}
    \end{figure}

    For this, we consider a certain subpath of $Q_{j_0}$: following the path $Q_{j_0}$ starting from $s$, let $z$ be the first vertex of $V(Y_1)\cup V(Y_2)$ reached by $Q_{j_0}$; such vertex exists because $u_{j_0}$ is a common endpoint of $Q_{j_0}$ and $P_{i_0}$, and $V(P_{i_0}) \subseteq V(Y_1) \cup V(Y_2)$.
    We denote by $Q'$ the subpath of $Q_{j_0}$ starting from $s$ and finishing at $z$.
    Let $j\in[2]$ such that $z\in V(Y_j)$. 
    We proceed to the construction of a model $L_k$ in $G$.

    We label the vertices $L_k$ by adding to $L_{k-1}$ a new column $(c_1,c_2)$ next to the first column of $L_{k-1}$, namely $(c^0_1,c^0_2)$.
    Then, for every $x \in V(L_k)$, let
    \[
    B_x =
    \begin{cases}
        B^1_x & \textrm{if $x \in V(L_{k-1})$,} \\
        \big(V(Y_j)\setminus \{u_j'\}\big) \cup V(Q') & \textrm{if $x = c_{j}$,} \\
        V(Y_{3-j})\setminus \{u_{3-j}'\} & \textrm{if $x = c_{3-j}$.}
    \end{cases}
    \]
    This gives the desired model of $L_k$ in $G$.
\end{proof}

\begin{proof}[Proof of \Cref{thm:rooted:tdtwo:ladder}]
    We set
    for every positive integer $k$
    \[
        f(k) = 
        \begin{cases}
            2 & \textrm{if $k=1$,} \\
            f_{\ref{lemma:td2:main_induction}}(k,2k-1)+1 & \textrm{otherwise.}
        \end{cases}
    \]

    Let $k$ be a positive integer,
    and let $(G,S)$ be an annotated graph with $\td_2(G,S) \geq f(k)$.
    By \cref{lem:S_B}, we can assume that $G$ has only one block, namely $G$ itself.
    If $k=1$, then clearly there is a vertex in $S$ which is not isolated,
    and so $(L_1,T_1)$ is a minor of $(G,S)$.
    Now suppose $k \geq 2$.
    Fix two vertices $u_1,u_2$ of $G$. 
    (This is possible since $2 \leq \td_2(G,S) \leq |V(G)|$.)
    By \Cref{lemma:td2:main_induction}, either
    \begin{enumerate}[label={(\alph*)}]
        \item there is a model of $(L_k,T_k)$ in $(G,S)$, or
        \item there is a model $(B_x \mid x \in V(L_{2k-1}))$ 
            of $L_{2k-1}$ in $G$ such that 
            \begin{enumerate}
                \item for every column $C$ of $L_{2k-1}$,
                    $S \cap (\bigcup_{c \in C}B_c) \neq \emptyset$, and
                \item if $\{c_1,c_2\}$ is the first column of $L_{2k-1}$,
                    then $u_1 \in B_{c_1}$ and $u_2 \in B_{c_2}$ or vice-versa.
            \end{enumerate}
    \end{enumerate}
    In the first case, we are done.
    Now suppose we are in the second case.
    Then, by the pigeonhole principle, 
    there is a row $R$ of $L_{2k-1}$ such that
    for at least $k$ vertices $x$ in $R$ we have 
    $S \cap B_x \neq \emptyset$. 
    Contracting in this model all the other columns with 
    one of their neighboring columns, 
    we find a model of $(L_k, T_k)$ in $(G,S)$.
\end{proof}


\subsection{Hadwiger number}

The \defin{Hadwiger number} of a graph $G$, denoted by \defin{$\had(G)$}, is the maximum number $k\in\mathbb{N}$ such that $G$ contains $K_k$ as a minor.
The main result of this subsection is the following theorem.
Together with \cref{obs:transfer-large-value}, \cref{thm:had} implies the Hadwiger number statement in \cref{thm:main}.

\begin{theorem}\label{thm:had}
    For every annotated graph $(G,S)$ and for every positive integer $k$, if $\had(G,S) \geq 2k$, 
    then $K_k$ is an $S$-rooted minor of $G$.
\end{theorem}

Note that such a statement with $2k$ replaced by $\Omega(k^{2300})$ follows from~\cite[Theorem~1.2]{ColorfulMinors}.
The exact statement of~\Cref{thm:had} follows from a result of Dujmović et al.~\cite[Corollary~13]{GMT-revisited}, which was inspired by the work of Kawarabayashi on rooted minors~\cite{K04}, which is itself based on \cite[Statement~5.4]{GMXIII} by Robertson and Seymour.
We include the proof of~\Cref{thm:had}, which is an adaptation of the \emph{generic folio technique} and, in particular, of the argument in~\cite{GMT-revisited}.

Given a graph $G$ and $F \subset \binom{V(G)}{2}$, 
we denote by \defin{$G \cup F$} the graph $(V(G), E(G) \cup F)$.

\begin{lemma}\label{lemma:appendix:rooted_model_or_separation}
    Let $k$ be a positive integer.
    Let $(G,S)$ be an annotated graph and let $\mathcal{M} = (U_x \mid x \in V(K_{2k}))$ be a model of $K_{2k}$
    in $G \cup \binom{S}{2}$ such that $S$ intersects at most $k$ branch sets in $\mathcal{M}$, i.e., $|\{x \in V(K_{2k}) \mid U_x \cap S \neq \emptyset\}| \leq k$.
    Then one of the following holds:
    \begin{enumerate}
        \item $K_k$ is an $S$-rooted minor of $G$, or \label{lem:item:rooted-model}
        \item there is a separation $(A,B)$ of $G$ of order at most $k-1$ such that
            $S \subseteq V(A)$, and there exists $z \in V(K_{2k})$ with $U_z \subseteq V(B) \setminus V(A)$. \label{lem:item:separation}
    \end{enumerate}
\end{lemma}

\begin{proof}
    Suppose that the statement does not hold, 
    and consider a counterexample $(G,S,\mathcal{M})$ with $|V(G)|$ minimum.

    \begin{claim}\label{claim:appendix:rooted_clique:t_separation}
        There is no separation $(A,B)$ of $G$ of order at most $k$ such that $S \subseteq V(A)$,
        $V(A) \setminus S \neq \emptyset$,
        and $U_z \subseteq V(B) \setminus V(A)$ for some $z \in V(K_{2k})$.
    \end{claim}

    \begin{proofclaim}
        Suppose that such a separation $(A,B)$ exists, and let $z$ be as in the statement.
        Since $G$, $S$, and $\mathcal{M}$ are a counterexample to the lemma, 
        there is no separation $(A',B')$ of $G$ of order at most $k-1$ with $S \subseteq V(A')$ and $V(B) \subseteq V(B')$.
        Therefore, by Menger's theorem, there is a family $(P_u \mid u \in V(A) \cap V(B))$
        of pairwise disjoint \stpaths{S}{(V(A) \cap V(B))} in $G$ with $u \in V(P_u)$ for every $u \in V(A) \cap V(B)$.
        Note that, for all $u \in V(A) \cap V(B)$, $P_u$ is a path in $A$.
        Now, for every $x \in V(K_{2k})$, let $U'_x = U_x \cap V(B)$. 
        Since $|V(A)\cap V(B)|\le k$ and $U_z\subseteq V(B)\setminus V(A)$, $U'_x$ is non-empty for every $x\in V(K_{2k})$. Note that every path in $G[A]$ between vertices in $V(A)\cap V(B)$ can be simulated by an edge between these two vertices and therefore $\mathcal{M}' = (U'_x \mid x \in V(K_{2k}))$ is a model of $K_{2k}$ in $B \cup \binom{V(A)\cap V(B)}{2}$.
        Note that 
        \[
            |\{x \in V(K_{2k}) \mid U_x' \cap V(A) \cap V(B) \neq \emptyset\}| \leq |V(A) \cap V(B)| \leq k.
        \]
        We next show that $|V(B)| < |V(G)|$, which implies that the statement of the lemma holds for $B$, $V(A) \cap V(B)$, and $\calM'$ because of minimality of the counterexample. To see why $|V(B)| < |V(G)|$ holds, consider a vertex $w\in V(A)\setminus S$; this exists because of the assumption in the statement of the claim. If $w\notin V(B),$ then we get directly that $|V(B)| < |V(G)|$. If $w\in V(A)\cap V(B)$, then the endpoint of $P_w$ in $S$ does not belong to $V(B)$ and hence we get $|V(B)| < |V(G)|$. Thus,  the statement of the lemma holds for $B$, $V(A) \cap V(B)$, and $\calM'$.

        If~\ref{lem:item:rooted-model} holds, then
        there is a $(V(A)\cap V(B))$-rooted model $(V_x \mid x \in V(K_k))$ of $K_k$ in $B$.
        For each $x \in V(K_k)$, let $u_x \in V(A) \cap V(B) \cap V_x$. 
        Then, $\big(V_x \cup V(P_{u_x}) \mid x \in V(K_k)\big)$ is an $S$-rooted model of $K_k$ in $G$ contradicting $G$, $S$, and $\mathcal{M}$ being a counterexample.

        If~\ref{lem:item:separation} holds, there is a separation $(A',B')$ of $B$ with $V(A)\cap V(B) \subseteq V(A')$
        and $U'_{z'} \subseteq V(B') \setminus V(A')$ for some $z' \in V(K_{2k})$.
        But then we have $U'_{z'} = U_{z'}$,
	    and $(A \cup A', B')$ is a separation of $G$ of order at most $k-1$,
        with $S \subseteq V(A \cup A')$
        and $U_{z'} \subseteq V(B') \setminus V(A \cup A')$,
        again contradicting the fact that
        $(G,S,\mathcal{M})$ is a counterexample.
    \end{proofclaim}

    \begin{claim}\label{claim:appendix:rooted_clique:singletons}
        For every $x \in V(K_{2k})$,
        either $U_x \subseteq S$, or $|U_x|=1$.
    \end{claim}

    \begin{proofclaim}
        Assume otherwise, and let $y \in V(K_{2k})$ be such that $U_y \not\subseteq S$ and $|U_y| > 1$.
        In particular, $G[U_y]$ has an edge $uv$ with $\{u,v\}\not\subseteq S$.
        Let $G'$ be obtained from $G$ by contracting $uv$ into a single vertex $w$,
        let 
        \[
            S' =
            \begin{cases}
                (S \setminus \{u,v\}) \cup \{w\} & \textrm{if $\{u,v\} \cap S \neq \emptyset$} \\
                S & \text{otherwise,}
            \end{cases}
        \]
        and let $\calM' = \{U_x' \mid x \in V(K_{2k})\}$ where,
        for every $x \in V(K_{2k})$, 
        \[
            U'_x = 
            \begin{cases}
                (U_x \setminus \{u,v\}) \cup \{w\} & \textrm{if $\{u,v\} \cap U_x \neq \emptyset$ (in other words, if $x=y$),} \\
                U_x & \textrm{otherwise.}
            \end{cases}
        \]
        Note that $\calM'$ is a model of $K_{2k}$ 
        in $G' \cup \binom{S'}{2}$ with 
        \[
            |\{x \in V(K_{2k}) \mid U_x' \cap S' \neq \emptyset\}| 
            = |\{x \in V(K_{2k}) \mid U_x \cap S \neq \emptyset\}| 
            \leq k.
        \]
        Since $|V(G')| < |V(G)|$, the statement of the lemma holds for $G'$, $S'$, and $\calM'$.

        The annotated graph $(G',S')$ is a minor of $(G,S)$, hence, if $G'$ contains $K_{k}$ as an $S'$-rooted minor, then $G$ also contains $K_{k}$ as an $S$-rooted minor.
        Thus, if~\ref{lem:item:rooted-model} holds for $(G',S',\calM')$, 
        then it also holds for $(G,S,\calM)$, 
        which contradicts the fact that it is a counterexample.

        If~\ref{lem:item:separation} holds for $(G',S',\calM')$, 
        then there is a separation $(A',B')$ of $G'$
        of order at most $k-1$ with $U'_z \subseteq V(B') \setminus V(A')$ for some $z \in V(K_{2k})$.
        Let $(A,B)$ be the separation of $G$ obtained from $(A',B')$
        by uncontracting the edge $uv$.
        Note that $(A,B)$ has order at most $k$.
        If $(A,B)$ has order at most $k-1$, 
        this contradicts the fact that $(G,S,\calM)$ is a counterexample.
        Thus, we suppose that $(A,B)$ is of order exactly $k$, and
        so $u,v \in V(A) \cap V(B)$.
        In particular, $V(A) \setminus S \neq \emptyset$.
        Since we also have $S \subseteq V(A)$ 
        and $U_z \subseteq V(B) \setminus V(A)$,
        this contradicts \Cref{claim:appendix:rooted_clique:t_separation}.
    \end{proofclaim}

    Let $T = \bigcup_{x \in V(K_{2k}), U_x \cap S = \emptyset} U_x$.
    By \Cref{claim:appendix:rooted_clique:singletons}, 
    $T$ induces a complete graph in $G$,
    which has size at least $k$ since 
    $|\{x \in V(K_{2k}) \mid U_x \cap S \neq \emptyset\}| \leq k$.
    By Menger's theorem,
    either there is a separation $(A,B)$ of $G$ of order at most $k-1$
    with $S \subseteq V(A)$ and $T \subseteq V(B)$ or 
    there are $k$ pairwise disjoint \stpaths{S}{T} $Q_1,\dots,Q_k$ in $G$.
    In the former case, there exists $u \in T \setminus (V(A) \cap V(B))$, 
    and so, 
    there exists $z \in V(K_{2k})$ with $U_z \subseteq V(B) \setminus V(A)$.
    This implies~\ref{lem:item:separation} for $G$, $S$, and $\calM$, 
    which is a contradiction.
    In the latter case, $(V(Q_i) \mid i \in [k])$ is an $S$-rooted model of $K_k$ in $G$,
    which is also a contradiction with $G$, $S$, and $\calM$ being a counterexample.
    This completes the proof of the lemma.
\end{proof}

\begin{corollary}\label{lemma:find_rooted_cliques}
    Let $k$ be a positive integer
    and let $(G,S)$ be an annotated graph with $|S|\leq k$ and $S \neq \emptyset$.
    If $K_{2k}$ is a minor of $G \cup \binom{S}{2}$,
    then, for some $\ell \in [k]$, 
    there is a separation $(A,B)$ in $G$ of order $\ell$ such that $V(B)\setminus V(A)\neq \emptyset$,
    $S \subseteq V(A)$, and $K_\ell$ is an $(V(A) \cap V(B))$-rooted minor of $B$.
\end{corollary}

\begin{proof}
    We proceed by induction on $k$.
    For $k=1$, one can take $A$ to be a $1$-vertex graph containing the vertex of $S$, and $B = G$. 
    Note that $V(B) \setminus  V(A)$ is non empty, 
    and the unique vertex of $A$ forms a $(V(A) \cap V(B))$-rooted
    model of $K_1$ in $B$.
    Now, assume that $k\geq 2$ and that the result holds for all positive integers less than $k$. 
    Let $\mathcal{M}=(U_x \mid x\in V(K_{2k}))$ be a model of $K_{2k}$ in $G \cup \binom{S}{2}$. 
    Since $|S|\le k$, at most $k$ sets in $\mathcal{M}$ intersect $S$.
    Apply then \Cref{lemma:appendix:rooted_model_or_separation} to $G,S,\mathcal{M}$. 
    If there is an $S$-rooted model of $K_k$ in $G$,
    then take $A$ to be the graph with vertex set $S$ with no edges
    and $B$ to be the whole graph $G$, and the lemma is satisfied with $\ell=k$ (note that $V(B) \setminus V(A) \neq \emptyset$ since \(|V(G)|\ge 2k>|S|\)).
    Otherwise, there exists a separation $(C,D)$ in $G$ of order at most $k-1$ 
    and $z \in V(K_{2k})$
    such that $S \subseteq V(C)$ and $U_z \subseteq V(D) \setminus V(C)$.
    Let $Z$ be the connected component of $D$ containing $U_z$. 
    Since $z$ is adjacent to every other vertex in $K_{2k}$, 
    $U_x$ contains a vertex of $Z$ for every $x\in V(K_{2k})$. 
    Let $\mathcal{M}_Z$ be obtained from $\mathcal{M}$ by replacing each branch set in $\mathcal{M}$ by its intersection with $V(Z)$. 
    Let $S' = V(C) \cap V(Z)$ and keep in mind that $|S'|\le k-1$. 
    Observe that $\mathcal{M}_Z$ is a model of $K_{2k}$ in $Z \cup \binom{S'}{2}$.
    By the induction hypothesis applied to $Z,S'$, and a restriction of $\mathcal{M}_Z$ to a model of $K_{2(k-1)}$,
    there is a separation $(A',B')$ in $Z$ of order $\ell:=|V(A')\cap V(B')|\le k-1$ such that $V(B')\setminus V(A')\neq \emptyset$,
    $S'\subseteq V(A')$, and $K_\ell$ is a $(V(A') \cap V(B'))$-rooted minor of $B'$.     
    Finally, put $A = C \cup A' \cup (D-V(Z))$ and $B = B'$.
    It follows that the separation $(A, B)$ has the desired properties. 
\end{proof}

\begin{proof}[Proof of \Cref{thm:had}]
    Suppose that there is no $S$-rooted model of $K_k$ in $G$. We will show that there is a set $S'$ with $S\subseteq S'\subseteq V(G)$ such that $\torso_G(S')$ does not contain $K_{2k}$ as a minor.
    We proceed by induction on $|V(G)|$.
    If $G$ is not connected, then we apply the induction hypothesis on each connected component and we are done.
    Now suppose that $G$ is connected.
    If $G$ does not contain $K_{2k}$ as a minor, then we are done because $\torso_G(V(G))=G$.
    Otherwise, let $\mathcal{M} = (U_x \mid x \in V(K_{2k}))$ be a model of $K_{2k}$ in $G$.
    We assume that $S$ intersects at most $k$ branch sets of $\calM$,
    since otherwise $K_k$ is an $S$-rooted minor of $G$.
    So, by \Cref{lemma:appendix:rooted_model_or_separation},
    there is a separation $(A,B)$ of $G$ with $S \subseteq V(A)$ and $U_z \subseteq V(B) \setminus V(A)$
    for some $z \in V(K_{2k})$.
    Then,
    since $z$ is adjacent to every other vertex in $K_{2k}$,
    the family $(U_x \cap V(B) \mid x \in V(K_{2k}))$ is a model of $K_{2k}$ in $B \cup \binom{V(A) \cap V(B)}{2}$.
    By \Cref{lemma:find_rooted_cliques},
    there is a separation $(C,D)$ of $B$ with $V(C) \cap V(D) \neq \emptyset$,
    $V(A) \cap V(B) \subseteq V(C)$, and such that
    there is a $(V(C) \cap V(D))$-rooted model of $K_{|V(C) \cap V(D)|}$ in $D$.
    In particular, this implies that $G' = \torso_G(V(C) \cup V(A))$ has no $S$-rooted model of $K_k$.
    So, by the induction hypothesis, there exists $S' \subseteq V(C)\cup V(A)\subseteq V(G)$ with $S \subseteq S'$
    such that $\torso_{G'}(S')$ is $K_{2k}$-minor-free.
    Since $\torso_{G'}(S') = \torso_G(S')$, this proves the theorem.
\end{proof}

\section{Erdős–Pósa property of lollipops}\label{sec:EP}

Let $G$ be a graph and $S \subset V(G)$.
An \defin{$S$-lollipop} in $G$ is a subgraph of $G$ that is the union of a cycle $C$ and a path $P$ such that $V(P)$ intersects both $V(C)$ and $S$.
Note that an annotated graph $(G,S)$ contains a lollipop if and only if
$(K_3,T)$ is a minor of $(G,S)$, where $T$ is a singleton in $V(K_3)$.
Therefore, the next theorem implies~\Cref{thm:EP-triangles-in-one-vx}.

\begin{theorem}\label{thm:EP-lollipops}
    There exists an absolute constant $c$ such that for every graph $G$, 
    for every $S\subseteq V(G)$, and
    for every positive integer $k$, 
    either
    \begin{enumerate}
        \item $G$ has $k$ disjoint $S$-lollipops or\label{item:EPS_thm_i}
        \item there exists a set $Z \subseteq V(G)$ such that $|Z| \leq c \cdot k\log k$ and $G-Z$ has no $S$-lollipops.\label{item:EPS_thm_ii}
    \end{enumerate}
\end{theorem}

We will use the following classical lemma of Simonovits~\cite{Simonovits1967}.
First, for every positive integer $k$, let
\[\defin{\simono(k)} = \begin{cases}
        \lceil 4k(\log k + \log \log k + 4)\rceil &\text{ if $k \geq 2$,}
        \\
        2 &\text{ if $k = 1$.}
        \end{cases}\]

Note that $\simono(k) \in \mathcal{O}(k\log k)$.

\begin{lemma}[\cite{Simonovits1967}]\label{lem:simonovits}
    Let $k$ be a positive integer and let $H$ be a graph with all vertices of degree $2$ or $3$. 
    If $H$ contains at least $\simono(k)$ vertices of degree $3$, then $H$ contains $k$ pairwise disjoint cycles.
\end{lemma}


We will use as a black-box the Erd\H{o}s-P\'osa Theorem.

\begin{theorem}[Erd\H{o}s-P\'osa Theorem~\cite{EP65}]\label{thm:EP:classic}
    There is a function $\ep \colon \mathbb{N} \to \mathbb{N}$ such that the following holds.
    For every graph $G$ and every positive integer $k$, 
    either
    \begin{enumerate}[label=(\arabic*)]
        \item $G$ has $k$ pairwise disjoint cycles, or
        \item there exists $Z \subseteq V(G)$ such that $|Z| \leq \ep(k)$ and $G-Z$ has no cycle.
    \end{enumerate}
    Additionally, $\ep(k) \in \mathcal{O}(k \log k)$.
\end{theorem}

In what follows, we fix such a function $\ep(\cdot)$.


\begin{proof}[Proof of~\Cref{thm:EP-lollipops}]
    For every positive integer $k$, let
        \[\alpha(k) = \left\lceil \frac{\simono(2k-1)}{2} \right\rceil + (k-1)\]
    and let $c$ be a constant such that for every integer $k$ with $k \geq 2$,
        \[\ep(k) + \alpha(k) \leq c \cdot k\log k. \]
    Let $G$ be a graph, let $S\subseteq V(G)$, and let $k$ be a positive integer.
    If $k=1$, then either $G$ contains an $S$-lollipop or not and the result holds. 
    Now suppose $k \geq 2$.
    First, we apply the Erdős–Pósa theorem (\Cref{thm:EP:classic}) to $G$ and $k$.
    This gives that either there are $k$ pairwise disjoint cycles in $G$, or there exists $Z \subset V(G)$ with $|Z| \leq \ep(k)\leq c \cdot k \log k$ such that $G - Z$ is a forest.
    In the latter case, $G - Z$ has no cycles, and so $G - Z$ has no $S$-lollipops.
    Thus, we assume the former.
    Let $\calC$ be the maximum family of pairwise disjoint cycles in $G$. 
    In particular, $|\calC| \geq k$.

    Let $\calP$ be a maximal family of pairwise vertex-disjoint \stpaths{V(\calC)}{S} in $G$ such that for each two paths in $\calP$, the endpoints in $V(\calC)$ lie in distinct cycles of $\calC$.
    Let $\calC_1$ be the family of all cycles in $\calC$ containing endpoints of paths in $\calP$.
    If $|\calC_1| \geq k$, then clearly $G[V(\calC_1 \cup \calP)]$ contains $k$ pairwise disjoint $S$-lollipops.
    Thus, we assume that $|\calC_1| < k$.
    In particular, $\calC \setminus \calC_1 \neq \emptyset$.

    We apply Menger's theorem between $V(\calC \setminus \calC_1)$ and $S$.
    We get that either there are $\alpha(k)$ pairwise disjoint \stpaths{V(\calC \setminus \calC_1)}{S} in $G$, or there exists a set $Z' \subset V(G)$ with $|Z'| < \alpha(k)$ such that $G - Z'$ has no \stpaths{V(\calC \setminus \calC_1)}{S}.
    First, assume that the latter holds and let $H$ be the subgraph of $G$ defined as the union of all components of $G - Z'$ containing a vertex in $S$.
    It follows that $H$ has no vertices of $V(\calC \setminus \calC_1)$.
    In particular, $H$ has no $k$ pairwise disjoint cycles, as otherwise these cycles, along with the cycles in $\calC \setminus \calC_1$ form a family of pairwise disjoint cycles in $G$ larger than $\calC$,
    which contradicts $\calC$ being maximum.
    Therefore, by the Erd\H{o}s–P\'osa theorem, there exists $Z'' \subset V(H)$ with $|Z''| \leq \ep(k)$ such that $H - Z''$ is a forest.
    In particular, for $Z = Z' \cup Z''$, we have $|Z| \leq \ep(k) + \alpha(k) \leq ck\log k$ and $G - Z$ has no $S$-lollipops.
    
    Therefore, we assume that there is a family $\calQ$ of at least $\alpha(k)$ pairwise disjoint \stpaths{V(\calC \setminus \calC_1)}{S} in $G$. We will show that this contradicts the choice of $\calC$.
    Let $\calC_2$ be the family of all cycles in $\calC\setminus\calC_1$ containing endpoints of paths in $\calQ$.
    Similarly as in the case of $\calC_1$, we may assume that $|\calC_2| < k$.

    Let $Q \in \calQ$.
    If $V(Q) \cap V(\bigcup \calP \cup \bigcup \calC_1) = \emptyset$, then $\calP \cup \{Q\}$ is a larger family of \stpaths{V(\calC)}{S} that are pairwise disjoint contradicting the definition of $\calP$.
    Thus, $V(Q) \cap V(\calP \cup \calC_1) \neq \emptyset$, and the next definition is justified.
    Let $q_Q$ be the endpoint of $Q$ in $V(\calC_2)$, and let $q_Q'$ be the first vertex in $V(\calC_1 \cup \calP)$ along $Q$ starting from $q_Q$.
    Note that $q_Q$ is distinct from $q'_Q$ since $V(\calC_2) \cap V(\calC_1 \cup \calP) = \emptyset$.
    Let $R_Q$ be the $q_Q$--$q_Q'$ subpath of $Q$.
    Next, we define a path $R_Q'$.
    If $q_Q' \in V(\calC_1)$, then we set $R_Q'$ to be the one-vertex path containing $q_Q'$.
    Otherwise, for $P \in \mathcal{P}$ such that $q'_Q \in V(P)$, 
    we set $R_Q'$ to be the $V(\calC_1)$--$q_Q'$ subpath of $P$.
    Let $\calQ'$ be the set of all $Q \in \calQ$ such that $q_Q' \notin V(\calC_1) \cap V(\calP)$.
    
    Define
    \[
        H = \bigcup(\calC_1 \cup \calC_2) \cup \bigcup_{Q \in \calQ'} (R_Q \cup R_Q').
    \]
    First, we claim that $H$ has only vertices of degree $2$ or $3$.
    All vertices of $\bigcup (\calC_1 \cup \calC_2)$ are of degree at least $2$ in $H$.
    By construction, for every $Q \in \calQ'$, the endpoints of $R_Q$ and $R_Q'$ in $G$ are in $H$, hence, all vertices of $H$ have degree at least $2$.
    On the other hand, recall that 
    \begin{enumerate}
        \item the cycles in $\calC_1 \cup \calC_2$ are pairwise disjoint;
        \item the paths in $\calP$ are pairwise disjoint;
        \item the paths in $\calQ$ are pairwise disjoint;
        \item the paths in $\calP$ are disjoint from the cycles in $\calC_2$; and
        \item for $Q\in \calQ'$, the vertex $q_Q'$ intersects a path of $\calP$ or a cycle of $\calC_1$, but not both by the definition of $\calQ'$.  
    \end{enumerate}
    Therefore we obtain that all vertices of $H$ have degree at most $3$.

    Next, we claim that $H$ has at least $\simono(2k-1)$ vertices of degree $3$.
    Note that since $|\calP| < k$ and the paths in $\calQ$ are pairwise disjoint, we have $|\calQ \setminus \calQ'| < k$.
    In particular, $|\calQ'| \geq \lceil \simono(2k-1) \slash 2\rceil$, or in other words, $2|\calQ'| \geq \simono(2k-1)$.

    Let $\calQ'_\calC = \{Q \in \calQ' \mid q_Q' \in V(\calC_1)\}$ and for each $P \in \calP$, let $\calQ_P' = \{Q \in \calQ' \mid q_Q' \in V(P)\}$.
    Clearly, $\{\calQ'_\calC\}\cup \{\calQ_P' \mid P \in \calP\}$ is a partition of $\calQ'$.
    Let $\calP' = \{P \in \calP \mid \calQ_P' \neq \emptyset\}$.
    For each $P \in \calP'$, let $p_P$ be the endpoint of $P$ in $V(\calC_1)$ and let $Q_P \in \calQ_P'$ be such that $q_{Q_P}'$ is the furthest vertex in $P$ starting from $p_P$ among $\{q_Q' \mid Q \in \calQ_P'\}$.
    It follows that the vertices in the following set have degree $3$ in $H$:
    \[
        \{q_Q \mid Q \in \calQ'\} 
        \cup \{q_Q' \mid Q \in \calQ_\calC'\} 
        \cup \bigcup_{P \in \calP'}(\{p_P\} 
        \cup \{q_Q' \mid Q \in \calQ_P' \setminus\{Q_P\}\}).
    \]
    Additionally, the above set has size
    \[
        |\calQ'| + |\calQ_\calC'| + \sum_{P \in \calP'} \big(1 + (|\calQ_P'|-1)\big) 
        = 2|\calQ'| 
        \geq \simono(2k-1),
    \]
    since $\{\calQ'_\calC\}\cup \{\calQ_P' \mid P \in \calP\}$ is a partition of $\calQ'$.
    It follows that indeed, $H$ has at least $\simono(2k-1)$ vertices of degree $3$.
    By~\Cref{lem:simonovits}, $H$ contains a family $\calC'$ of $2k-1$ pairwise disjoint cycles.
    Since the paths in $\calP \cup \calQ'$ are disjoint from 
    the cycles in $\calC \setminus (\calC_1\cup \calC_2)$ and
    $H \subset \bigcup (\calC_1\cup \calC_2 \cup \calP \cup \calQ)$, 
    the family $\calD = \big(\calC \setminus (\calC_1 \cup \calC_2)\big) \cup \calC'$ 
    is a family of pairwise disjoint cycles in $G$.
    Moreover, 
    \[|\calD| = |\calC| - |\calC_1| - |\calC_2| + |\calC'| \geq |\calC| - 2(k-1) + 2k-1 = |\calC| + 1.\]
    This is a contradiction with the maximality of $\calC$, which ends the proof.
\end{proof}

\section{Open problems}\label{sec:open}

To conclude, we would like to gather some indicative conjectures and open problems that we consider relevant to the study of annotated graphs.

We conjecture that for plane annotated graphs where the annotated set appears in the outer-face, the Erd\H{o}s-P\'osa function is $\mathcal{O}(k \log k)$.


\begin{conjecture}[Erd\H{o}s-P\'osa function of annotated minors]\label{conj:rooted_EP_k_logk}
    Let $(H,T)$ be an annotated graph with $H$ a plane graph
    and $T$ contained in its outer-face.
    There exists a constant $c(H,T)$ such that,
    for every annotated graph $(G,S)$,
    for every positive integer $k$, either
    \begin{enumerate}
        \item $(G,S)$ has $k$ disjoint models of $(H,T)$, or
        \item there exists $Z \subseteq V(G)$ with $|Z| \leq c(H,T) \cdot k \log k$
            such that $(G-Z,S \setminus Z)$ has no model of $(H,T)$.
    \end{enumerate}
\end{conjecture}
Similar to the standard (non-annotated) case, the bound in $\bigO(k \log k)$ would be tight for every $(H,T)$ 
where $H$ is not a forest (see e.g.~\cite{EP65,tightEP}).
As mentioned earlier, a bound in $\bigO(k)$ in the case where $H$ is a tree was recently obtained
by Claus, Joret, Rambaud, and Robinson~\cite{EProotedtrees}.
In \Cref{thm:EP-lollipops}, 
we solve the case where $H=K_3$ and $T$ is a single vertex of $K_3$.
Therefore, the first open case of \Cref{conj:rooted_EP_k_logk}
is $H=K_3$ with $T$ consisting of two vertices of $H$.


Another intriguing line of research consists in finding the
minimal annotated obstructions of some classes of annotated graphs.
Probably the most natural example is the case of planar graphs:
\begin{problem}[Annotated Kuratowski] \label{annotated-Kuratowski}
    Determine the smallest set $\mathcal{O}$
    of annotated graphs such that,
    for every annotated graph $(G,S)$,
    there exists $S' \subseteq V(G)$ with $S \subseteq S'$ and $\torso_G(S')$ planar
    if and only if, for every $(H,T) \in \mathcal{O}$, $(H,T)$ is not a minor of $(G,S)$.
\end{problem}
Note that, by (the labeled variant of) the Robertson-Seymour Theorem,
such a set $\mathcal{O}$ exists and is finite.
In \Cref{thm:obstructions-forests}, we solved the analogous problem
for forests instead of planar graphs.


Also, it would be interesting to see whether, similar to the \textsc{Steiner Tree} problem~\cite{JS24}, more problems on terminal sets can be shown to be FPT parameterized not just by the size of the terminal set $S$ but rather by $\tw(G,S)$ or other torso parameters.

\section{Acknowledgments}
Some parts of this research were conducted at the Twelfth and Thirteenth Annual Workshop on Geometry and Graphs held at the Bellairs Research Institute in February 2025 and February 2026, respectively. Thanks to the organizers and the other workshop participants for creating a productive working atmosphere.

\bibliographystyle{plain}
\bibliography{bibliography}
\end{document}